\documentclass[11pt,thmsa,a4paper,reqno]{article}

\usepackage{amssymb}
\usepackage{amsthm}
\usepackage{amsxtra}
\usepackage{mathrsfs}
\usepackage{upgreek}
\usepackage{bbm}
\usepackage{mathtools}
\usepackage{stmaryrd}

\usepackage{extarrows}
\usepackage{enumitem}
\usepackage{tikz-cd}
\usepackage{url}
\usepackage{booktabs}
\usepackage{graphicx}
\usepackage{resizegather}
\usepackage{arydshln}
\usepackage{floatrow}
\usepackage{soul}
\usepackage{relsize,etoolbox}
 \usepackage[all,cmtip,2cell]{xy}
\UseAllTwocells
\xyoption{2cell}

\usepackage{eucal}

\numberwithin{equation}{section}

\usepackage[scr,scaled=1.1]{rsfso}

\usepackage[pdftex,
                paper=a4paper,
                portrait=true,
                textwidth=165mm,
                textheight=235mm,
                tmargin=2.5cm,
                marginratio=1:1]{geometry}
                
\usepackage{pdfrender}

\makeatletter
\newtheorem*{rep@theorem}{\rep@title}
\newcommand{\newreptheorem}[2]{%
\newenvironment{rep#1}[1]{%
 \def\rep@title{#2 \ref{##1}}%
 \begin{rep@theorem}}%
 {\end{rep@theorem}}}
\makeatother

\newtheorem*{theoremO}{Theorem}

\newtheorem{theorem}{Theorem}[section]
\newreptheorem{theorem}{Theorem}
\newtheorem{proposition}[theorem]{Proposition}
\newtheorem{lemma}[theorem]{Lemma}

\newreptheorem{corollary}{Corollary}

\theoremstyle{definition}
\newtheorem{definition}[theorem]{Definition}
\newtheorem{example}[theorem]{Example}

\theoremstyle{remark}
\newtheorem{remark}[theorem]{Remark}

\newtheorem*{acknowledgements}{Acknowledgements}

\newcommand{\RR}{\ensuremath{\mathbbmss{R}}}
\newcommand{\TT}{\ensuremath{\mathbbmss{T}}}

\def\Dscr{\mathscr{D}}
\def\Iscr{\mathscr{I}}

\newcommand{\gfrak}{\mathfrak{g}}

\newcommand{\Ad}{\operatorname{Ad}}

\newcommand{\End}{\operatorname{End}}
\newcommand{\Hom}{\operatorname{Hom}}

\newcommand{\CE}{\operatorname{CE}}

\newcommand{\GL}{\operatorname{GL}}

\newcommand{\Rep}{\operatorname{\mathbf{Rep}}}

\newcommand{\Mod}{\operatorname{\mathbf{Mod}}}

\def\uC{{\mathrm C}}

\def\uS{{\mathrm S}}

\def\id{{\rm id}}

\def\us{{\sf s}}
\def\uu{{\sf u}}

\def\AW{{\sf{AW}}}
\def\EZ{{\sf{EZ}}}

\let\originalleft\left
\let\originalright\right
\renewcommand{\left}{\mathopen{}\mathclose\bgroup\originalleft}
\renewcommand{\right}{\aftergroup\egroup\originalright}

\DeclareFontFamily{U}{matha}{\hyphenchar\font45}
\DeclareFontShape{U}{matha}{m}{n}{
      <5> <6> <7> <8> <9> <10> gen * matha
      <10.95> matha10 <12> <14.4> <17.28> <20.74> <24.88> matha12
      }{}
\DeclareSymbolFont{matha}{U}{matha}{m}{n}
\DeclareFontSubstitution{U}{matha}{m}{n}
\DeclareMathSymbol{\abxcup}{\mathbin}{matha}{'131}

\newcommand*{\sbullet}{\raisebox{0.1ex}{\scalebox{0.6}{$\bullet$}}}

\DeclareMathAlphabet{\mathbfit}{OML}{cmm}{b}{it}

\begin{document}

\title{Singular chains on Lie groups and the Cartan relations I}

\author{Camilo Arias Abad\footnote{Escuela de Matem\'{a}ticas, Universidad Nacional de Colombia Sede Medell\'{i}n, email: carias0@unal.edu.co}}

 \maketitle
 
 \begin{abstract}
Let $G$ be a simply connected Lie group with Lie algebra $\gfrak$. We show that the following categories are naturally equivalent. The category $\Mod(\uC_{\sbullet}(G))$ of sufficiently smooth modules over the DG algebra of singular chains on $G$. The category $\Rep(\TT\gfrak)$ of representations of the DG Lie algebra $\TT\gfrak$, which is universal for the Cartan relations. This equivalence extends the correspondence between representations of $G$ and representations of $\gfrak$. In a companion paper, we will show that in the compact case, the equivalence can be extended to a quasi-equivalence of DG categories.
\end{abstract}

\tableofcontents

\section{Introduction}
Let $G$ be a simply connected Lie group with Lie algebra $\gfrak$. We will denote by $\uC_{\sbullet}(G)$ the space of smooth singular chains on $G$. This space has the structure of a DG Hopf algebra with product induced by the Eilenberg-Zilber map and coproduct induced by the Alexander-Whitney map. In particular, $\uC_{\sbullet}(G)$ is an algebra whose category of modules is a tensor category. We denote by $\Mod(\uC_{\sbullet}(G))$ the category of sufficiently smooth modules over this algebra. An object $V \in \Mod(\uC_{\sbullet}(G))$ receives an action of the algebra of singular chains on $G$. In particular, it receives an action of the zero simplices of $G$ and is therefore a representation of $G$. Thus, $V$ is a cochain complex on which not only the points of $G$ but all simplices of $G$ act.

We address the problem of describing this category infinitesimally. We consider the DG-Lie algebra $\TT\gfrak$ which is universal for the Cartan relations, and its category of representations $\Rep(\TT\gfrak)$. Representations of this algebra are sometimes known as $\gfrak$-DG spaces in the literature. They arise in Chern-Weil theory and the infinitesimal models for equivariant cohomology \cite{GS, AM}.

Our main result is the following.

\begin{theoremO}\label{thm:main}
Let $G$ be a simply connected Lie group with Lie algebra $\gfrak$. There is a a differentiation functor \[\Dscr\colon \Mod(\uC_{\sbullet}(G)) \rightarrow \Rep(\TT\gfrak)\] and an integration functor \[\Iscr\colon \Rep(\TT\gfrak) \rightarrow \Mod(\uC_{\sbullet}(G))\]  which are inverse to one another. In particular, the categories $\Mod(\uC_{\sbullet}(G))$ and $\Rep(\TT\gfrak)$ are naturally equivalent as monoidal categories.
\end{theoremO}

We show that the forgetful functor $F\colon \Rep(\TT\gfrak) \rightarrow \Rep(\gfrak)$ admits a left adjoint $U\colon \Rep(\gfrak) \rightarrow \Rep(\TT\gfrak)$ which embeds the category of representations of $\gfrak$  in the category $\Rep(\TT\gfrak)$. The integration functor mentioned above makes the following diagram commute:
\[ \xymatrix{ \Rep(\gfrak) \ar[r]^{\mathsf{Lie}^{-1}} \ar[d]_-{U} & \Rep(G) \ar[d]^-{U}\\
\Rep(\TT\gfrak)\ar[r]^-{\Iscr} & \Mod(\uC_{\sbullet}(G))
}\]
\noindent
Thus, the construction in Theorem \ref{thm:main} extends the usual correspondence between representations of $\gfrak$ and representations of $G$.

\subsection*{Relationship with higher local systems and Chern-Weil theory}
The equivalence of categories $\Rep(\TT\gfrak) \cong \Mod(\uC_{\sbullet}(G))$ can be interpreted in terms of Chern-Weil theory for $\infty$-local systems. A local system on a topological space $X$ is a representation of the fundamental groupoid of $X$. Recently there has been interest in a theory of higher dimensional local systems where the fundamental groupoid is replaced by the $\infty$-groupoid, $\pi_\infty(X)$, of $X$. See for instance \cite{Holstein1,BS,AS2013,AS2016}. These $\infty$-local systems can be described in different ways:
\begin{center}
  \begin{tabular}{| c |  c | c | }
\hline 
   & {\bf Point of view}   & {\bf $\infty$-local system}   \\ \hline
   { $\boldsymbol{\mathsf{I}}$ }& Infinitesimal & Flat $\mathbb{Z}$-graded connection \\ 
{$\boldsymbol{\mathsf{S}}$} & Simplicial & Representation of $\pi_\infty(X)$\\ 
{$\boldsymbol{\mathsf{T}}$} &  Topological & Representation of $\uC_{\sbullet}(\Omega^{\mathrm{M}}(X)) $ \\   
\hline
  \end{tabular}
\end{center}
Here $\Omega^{\mathrm{M}}(X)$ denotes the based Moore loop space of $X$. Each of these notions of $\infty$-local system can be organized into a DG category. The resulting DG categories have been proven to be quasi-equivalent. The equivalence of $\boldsymbol{\mathsf{I}}$ and $\boldsymbol{\mathsf{S}}$ was proved by Block-Smith \cite{BS} extending ideas of Igusa \cite{Igusa}. The equivalence between $\boldsymbol{\mathsf{S}}$ and $\boldsymbol{\mathsf{T}}$ was proved by Holstein \cite{Holstein1}, using the homotopy theory of DG categories. 

One can consider the case $X= BG$ and take the point of view $\boldsymbol{\mathsf{T}}$ of local systems as modules over $\uC_{\sbullet}(\Omega^{\mathrm{M}}(X))$. Since $\Omega^{\mathrm{M}}(BG)$ is homotopy equivalent to $G$, one can think of the category $\Mod(\uC_{\sbullet}(G))$ as describing $\infty$-local systems on $BG$. Thus, Theorem \ref{thm:main} may be thought of as providing an infinitesimal description of the category of $\infty$-local systems on the classifying space of $G$. If $G$ is further assumed to be compact, stronger results can be obtained. The categories $\Rep(\TT\gfrak)$ and $\Mod(\uC_{\sbullet}(G))$ can be naturally enhanced to DG categories which are quasi-equivalent.
In particular, when this quasi-equivalence is applied to the endomorphisms of the trivial local system, one recovers the standard computation from Chern-Weil theory $\mathrm{H}^{\sbullet}(BG) \simeq (\uS^{\sbullet}\gfrak^*)^G$. Indeed, the Chern-Weil homomorphism can be extended to a Chern-Weil functor for any principal $G$-bundle with connection. The details of these constructions will be explained in \cite{CA,CAII}.

The paper is organized as follows:

\begin{itemize}
\item In  \S \ref{Prelim} we fix our conventions and collect some preliminaries regarding Chevalley-Eilenberg complexes, DG Hopf Algebras and the algebraic structure on the algebra of singular chains on a Lie group.

 \item  Our main results are proved in \S \ref{equiv}. Theorem \ref{dif} describes the differentiation functor \[\Dscr\colon \Mod(\uC_{\sbullet}(G)) \rightarrow \Rep(\TT\gfrak).\]  Theorem \ref{int} describes the integration functor \[\Iscr\colon \Rep(\TT\gfrak) \rightarrow \Mod(\uC_{\sbullet}(G)).\] Theorem \ref{main} states that these functors are equivalences. We also prove Theorem \ref{nosc}, which is a version of the equivalence of categories in the case where $G$ is not necessarily simply connected.
 
 \item In  \S \ref{examples} we discuss examples of the constructions above. We prove Proposition \ref{K}  which describes the left adjoint to the forgetful functor  $F\colon \Rep(\TT\gfrak) \rightarrow \Rep(\gfrak)$.
  
 \item Appendix \ref{A1} is concerned with a variation of the characterization of the image of the De Rham map via synthetic geometry by A. Kock \cite{Kock}.
 
\end{itemize}

\begin{acknowledgements}
We would like to acknowledge the support of Colciencias through  their grant {\it Estructuras lineales en topolog\'ia y geometr\'ia}, with contract number FP44842-013-2018.  We also thank  the Alexander von Humboldt foundation which supported our work through the Humboldt Institutspartnerschaftet { \it Representations of Gerbes and higher holonomies}. 
We are grateful to Anders Kock for pointing us to his book \cite{Kock}, where the image of the De Rham map is described.
We thank Alexander Quintero V\'elez and Eckhard Meinrenken for their important comments on a previous draft. We also thank Konrad Waldorf for his hospitality during a visit to Greifswald, where part of this work was completed. \end{acknowledgements}


\section{Conventions and preliminaries}\label{Prelim}
The following conventions are in place throughout the paper:
\begin{itemize}
\item All vector spaces and algebras are defined over the field $\RR$ of real numbers. The category of graded vector spaces is symmetric monoidal with respect to the natural isomorphism:
\[ \tau\colon V \otimes W \rightarrow W \otimes V, \quad v \otimes w\mapsto (-1)^{\vert v\vert \vert w \vert} w \otimes v.\]
\item Given a graded vector space $V$, the suspension of $V$, denoted $sV$, is the graded vector space such that
\[ (\us V)^k= V^{k-1}.\]
 The unsuspension of $V$, denoted $uV$, is the graded vector space such that
\[ (\uu V)^k= V^{k+1}.\]
\item Differentials on complexes, DG-algebras and DG-Lie algebras increase the degree by one.
\item Tensor products are taken over the field $\RR$ of real numbers unless otherwise stated.
\item Given a graded vector space $V$, we denote by $\uS^{\sbullet}V$ the free graded commutative algebra generated by $V$.
\item The commutator symbol $[a,b]$ denotes the graded commutator $ab-(-1)^{\vert a\vert \vert b \vert}ba$.
\item The Lie algebra $\gfrak$ of a Lie group $G$ is the Lie algebra of left invariant vector fields on $G$.
\item We will denote by $\Delta_k$ the standard geometric $k$-simplex
\[ \Delta_k =\{(t_1,...,t_k)\in \mathbb{R}^{k}: 1 \ge t_1 \ge t_2 \ge \cdots \ge t_k \ge 0\} .\]
The $i$-th vertex of $\Delta_k$ is
\[ v_i = (\underbrace{1, \dots, 1}_{i\text{ times}}, 0, \dots, 0).\]
The $i$-th face  of $\Delta_k$ is the convex closure of the vertices $\{ v_0, \dots \hat{v}_{n-i}, \dots ,v_{k}\}$. For each $i\leq k$ we define by $f_i^k\colon \Delta_i \rightarrow \Delta_k$ the inclusion of the $i$-simplex as the lowest $i$-th face of the
$k$-simplex, namely, the unique affine map that sends the $j$-th vertex of $\Delta_i$ to the $j$-th vertex of $\Delta_k$.
Similarly, we define by $b_i^k\colon \Delta_i \rightarrow \Delta_k$ the inclusion of the $i$-simplex as the highest $i$-th face of the $k$-simplex, namely, the unique affine map that sends the $j$-the vertex of $\Delta_i$ to the $j+k-i$-th vertex of $\Delta_k$.
\item If $V$ is a graded vector space, then the space of $V$-valued differential forms on a manifold $M$ is
\[ \Omega^{\sbullet}(M,V)=\Omega^{\sbullet}(M) \otimes V.\]
\end{itemize}

\subsection{Chevalley-Eilenberg complexes and the DG-Lie algebra $\TT\gfrak$}
Let $\gfrak$ be a finite dimensional Lie algebra. The Chevalley-Eilenberg algebra of $\gfrak$ is the differential graded algebra $\CE(\gfrak)$ defined as follows. As an algebra:
\[ \CE(\gfrak)=  (\uS^{\sbullet}(\uu\gfrak))^* \cong \uS^{\sbullet}(\uu\gfrak)^*,\]
is the free commutative algebra generated by the dual of the vector space $\uu\gfrak$. The differential is given by
\begin{eqnarray*} (\delta_{\CE}\eta)(v_1,\dots, v_{k+1})=\sum_{i<j}(-1)^{i+j}\eta([v_i,v_j], \dots ,\widehat{v}_i,\dots , \widehat{v}_j,\dots , v_{k+1}).
\end{eqnarray*}
Given a representation $V$ of $\gfrak$, there is a Chevalley-Eilenberg complex with values in $V$,  $\CE(\gfrak,V)=V\otimes \CE(\gfrak)$ with differential given by
\begin{eqnarray*} (\delta \omega)(v_1,\dots, v_{k+1})=\sum_{i<j}(-1)^{i+j}\omega([v_i,v_j], \dots ,\widehat{v}_i,\dots , \widehat{v}_j,\dots , v_{k+1}).\\+\sum_{i}(-1)^{i+1}\rho(v_i) \omega(v_1,\dots \widehat{v}_i,\dots ,v_{k+1}).
\end{eqnarray*}
The vector space $\CE(\gfrak,V)$ is a module over $\CE(\gfrak)$ and the differential is a derivation with respect to the module structure, that is,
\[ \delta(\eta \omega)= (\delta_{\CE}\eta) \omega +(-1)^{\vert \eta \vert}\eta (\delta\omega).\]

Given a representation $V$ of $\gfrak$ there is also a homology chain complex $\uC_{\sbullet}(\gfrak,V)$ which is defined as follows. As a vector space, 
\[ \uC_{\sbullet}(\gfrak,V)=\uS^{\sbullet}(\uu\gfrak)\otimes V.\]
The differential is given by the formula:
\begin{eqnarray*}
 \delta(x_1 \wedge \dots \wedge x_k \otimes v)&=&\sum_{i<i}(-1)^{i+j+1}[x_i,x_j] \wedge \dots \widehat{x}_i \wedge \dots \wedge \widehat{x}_j \wedge \dots x_k \otimes v+\\
 &&\sum_i (-1)^{i+1} x_1 \wedge \dots \wedge \widehat{x}_i \wedge \dots \wedge x_k \otimes \rho(x_i)(v). 
\end{eqnarray*}

\begin{definition}
Let $\gfrak$ be a Lie algebra. The differential graded Lie algebra $\TT\gfrak$ is defined as follows.
As a vector space
\[ \TT\gfrak = \uu \gfrak\oplus \gfrak.\]
Given $x\in \gfrak$ we denote by $L_x$ the corresponding element in $\TT\gfrak^0$ and by $i_x$ the corresponding element in $\TT\gfrak^{-1}$. The differential is given by $d(i_x) =L_x$. The bracket in $\TT\gfrak$ is given by the Cartan relations
\begin{align}
[L_x,L_y]&=L_{[x,y]},\\
[L_x,i_y]&=i_{[x,y]},\\
[i_x,i_y]&=0.
\end{align}
We will denote by $\Rep(\TT\gfrak)$ the category of representations of $\TT\gfrak$ on finite dimensional cochain complexes.
\end{definition}

\subsection{Singular chains on Lie groups and DG-Hopf algebras}\label{chains}
A differential graded Hopf algebra is a graded vector space $A$ together with a product map $m: A \otimes A \rightarrow A$, a coproduct map $\Delta\colon A \rightarrow A \otimes A$, a unit map $u\colon \RR \rightarrow A$, a counit map $c\colon A \rightarrow \RR$, an antipode $s\colon A \rightarrow A$  and a differential $d\colon A \rightarrow A$, such that
\begin{itemize}
\item $(A,m,u,d)$ is a DG algebra.
\item $(A, \Delta,c,d)$ is a DG coalgebra.
\item $c \circ u= \id_{\RR}$.
\item $m \circ \Delta= (m \otimes m) \circ ( \id \otimes \tau \otimes \id)\circ (\Delta \otimes \Delta)$.
\item $c$ is an algebra map.
\item $u$ is a coalgebra map.
\item $m \circ (\id \otimes s )\circ \Delta=u \circ c$.
\item $m \circ (s \otimes \id) \circ \Delta= u \circ c$.
\end{itemize}
We say that $A$ is commutative if $m \circ \tau=m$ and it is cocommutative if $ \tau \circ \Delta=\Delta$. 

Given a DG Hopf algebra $A$, the tensor product $V \otimes W$ of DG $A$-modules has the structure of an $A$-module via the map:
\[ \rho_{V \otimes W}\colon A \rightarrow \End(V \otimes W),\quad \rho_{V \otimes W}(a)= \mathsf{nat} \circ (\rho_V \otimes \rho_W) \circ \Delta,\]
where $\mathsf{nat}$ is the natural map
\[ \mathsf{nat}\colon \End(V) \otimes \End(W) \rightarrow \End(V \otimes W),\]
given by
\[ \mathsf{nat}(\phi \otimes \varphi)(v \otimes w) = (-1)^{\vert v\vert \vert \varphi \vert}\phi(v) \otimes \varphi(w).\]
In this way, the category of DG modules over a DG Hopf algebra acquires the structure of a monoidal category. If $A$ is cocommutative then this monoidal category is symmetric. If $V$ is a module over a DG Hopf algebra $A$, the dual complex $V^*$ is also a module over $A$ with structure given by
\[[\rho_{V^*}(a)(\phi)](v) =\phi( \rho_V(s(a))(v)).\]
This dual representation has the property that the natural pairing
\[ V \otimes V^* \rightarrow \RR\]
is a morphism of modules.

\begin{example}[{\bf Trivial module}]
Let $V$ be a cochain complex of finite dimensional vector spaces. Then $V$ can be seen as a trivial DG-module over $A$ via the composition
$$
\xymatrix{
A \ar[r]^-{c}&
\RR
 \ar[r]^-{u} &
\End(V).
}
$$
where $c$ is the counit of $A$ and $u$ is the unit of $\End(V)$.
\end{example}

\begin{example}
Let $\mathfrak{h}$ be a DG Lie algebra. The universal enveloping algebra $U(\frak{h})$ is a cocommutative DG Hopf algebra with antipode and coproduct determined by the conditions $s(x)=-x$ and
 $\Delta(x)=1 \otimes x + x \otimes 1$ for $x \in \mathfrak{h}$. It is always cocommutative and commutative if and only if $\mathfrak{h}$ abelian.
\end{example}

Given a Lie group $G$, we will denote by $\uC_{\sbullet}(G)$  the cochain complex of smooth singular chains on $G$. This space has the structure of a DG-algebra with product map \[m: \uC_{\sbullet}(G) \otimes \uC_{\sbullet}(G) \rightarrow \uC_{\sbullet}(G)\] given by the composition
$$
\xymatrix{
\uC_{\sbullet}(G)\otimes \uC_{\sbullet}(G) \ar[rr]^(0.55){\EZ}&&
\uC_{\sbullet}(G \times G) \ar[rr]^(0.6){\mu_*} &&
\uC_{\sbullet} (G).
}
$$
Here $\mu\colon G \times G \rightarrow G$ is the multiplication map, and $\EZ$ is the Eilenberg-Zilber map defined by
\begin{eqnarray*}
 \EZ(\sigma \otimes \nu) &=& \sum_{\chi \in \mathfrak{S}_{r,s}}(-1)^{|\chi|} (\sigma \times \nu)\circ \chi_*,
\end{eqnarray*}
where $\sigma \in \uC_r(G)$, $\nu \in \uC_s(G)$ and $\mathfrak{S}_{r,s}$ is the set of all shuffle permutations and $\chi_*$ is the map
\[ \chi_*\colon\Delta_{r+s}\to \Delta_r \times \Delta_s, \quad \chi_*(t_1,\dots,t_{r+s})=((t_{\chi(1)},\dots, t_{\chi(r)}),(t_{\chi(r+1)},\dots, t_{\chi(r+s)})).\]
The map $ \chi_*$ is a homeomorphism onto its image. Moreover, the union over all shuffle permutations of the images of the maps $ \chi_*$ cover $\Delta_r \times \Delta_s$ and the intersection of two different images has measure zero. 

The inversion map $\iota\colon G \to G$ induces a map:
\[ s=\iota_*\colon \uC_{\sbullet}(G) \to \uC_{\sbullet}(G).\]
There is also a coproduct $\Delta: \uC_{\sbullet}(G) \rightarrow \uC_{\sbullet}(G) \otimes \uC_{\sbullet}(G)$ defined as the composition
\[\xymatrix{
\uC_{\sbullet}(G) \ar[rr]^(0.45){\mathsf{diag}_*}&&
\uC_{\sbullet}(G \times G) \ar[rr]^(0.45){\AW} &&
\uC_{\sbullet} (G) \otimes \uC_{\sbullet}(G),
}\]
where $\mathsf{diag_*}$ is the map induced by the diagonal map and $\AW$ is the Alexander-Whitney map. This is the map
\[ \AW\colon \uC_{\sbullet}(X \times Y ) \rightarrow \uC_{\sbullet}(X) \otimes \uC_{\sbullet}(Y)\]
defined as follows. If $\sigma=(\sigma_1, \sigma_2)$ is a $k$-simplex in $X \times Y$ then:
\[ \AW(\sigma)=\sum_{i+j=k} (\sigma_1 \circ f^k_i) \otimes (\sigma_2 \circ b^k_i),\]
where $f^k_i$ and $b^k_i$ are the highest and lowest simplex maps defined above. 

\begin{example}
The space $\uC_{\sbullet}(G)$  has the structure of a DG Hopf algebra. The product, coproduct and antipode are the operators $m$, $\Delta$ and $s$ defined above. The unit is induced by the inclusion of a point as the identity. The counit is given by the map induced by the projection to a point. In general, this Hopf algebra is neither commutative nor cocommutative.
\end{example}

\begin{definition}
We will say that a simplex $\sigma\colon \Delta_k \rightarrow M$ is thin if the rank of the differential of $\sigma$ is less than $k$ at every point $p \in \Delta_k$. 
\end{definition}

\begin{definition}
We will denote by $\Mod(\uC_{\sbullet}(G))$ the category of left DG modules $\rho\colon \uC_{\sbullet}(G) \rightarrow \End(V)$ over the algebra $\uC_{\sbullet}(G)$ with the following properties:
\begin{enumerate}
\item[(a)] $V$ is a complex of finite dimensional vector spaces.
\item[(b)] The map $\rho\colon \uC_{\sbullet}(G) \rightarrow \End(V)$ vanishes on thin simplices.
\item[(c)] The map $\rho$ is smooth. That is, given a manifold $M$ and a smooth map
$\sigma_M\colon M \times \Delta_k \rightarrow G$,
the corresponding map
$\rho(\sigma_M)\colon M \rightarrow \End(V),$
given by
$ \rho(\sigma_M)(p)= \rho( \sigma_M(p,-))$
is smooth.
\item[(d)] Given a basis $\{ v_1, \dots, v_k\}$ of $V$, the corresponding cochains $a_{ij}$ are alternating and subdivision invariant in the sense of the Appendix \ref{A1}.
\end{enumerate}
\end{definition}

\section{An equivalence of categories of representations}\label{equiv}
In this section we prove our main results. We construct a functor $\Dscr\colon\Mod(\uC_{\sbullet}(G)) \rightarrow\Rep(\TT\gfrak)$, a functor $\Iscr\colon \Rep(\TT\gfrak) \rightarrow \Mod(\uC_{\sbullet}(G))$, and prove that they are inverses to one another.


\subsection{The differentiation functor}
Let $x_1, \dots, x_k$ be elements of the Lie algebra $\gfrak$ of the group $G$. We denote by $\sigma[x_1, \cdots, x_k]$
the simplex $\sigma[x_1, \dots,x_k]\colon \Delta_k \rightarrow G$ given by
\[ \sigma[x_1, \dots,x_k](t_1,\dots,t_k)= \exp(t_1 x_1) \cdots \exp(t_k x_k).\]

\begin{lemma}\label{square}
For any $x \in \gfrak$ the chain $\sigma[x]$ satisfies
\[ \sigma[x]^2=0.\]
\end{lemma}

\begin{proof}
The chain $ \sigma[x]^2$ is the difference between the chains
$\alpha(t,s)=\sigma[x](t) \sigma[x](s)$ and
$\beta(t,s)=\sigma[x](s) \sigma[x](t).$ These are equal, as the following computation shows:
$$ \alpha(t,s)=\exp(tx) \exp(sx)=\exp(sx) \exp(tx) =\beta(t,s). \qedhere$$
\end{proof}

\begin{remark}
Since $\uC_{\sbullet}(G)$ is a DG Hopf algebra, the category $\Mod(\uC_{\sbullet}(G))$ is monoidal. Also, since $U(\TT\gfrak)$ is a cocommutative DG Hopf algebra, the category $\Rep(\TT\gfrak)$ is symmetric monoidal. Explicitly, if $V, V' \in \Rep(\TT\gfrak)$ then $V \otimes V'$ is also a representation with action given by
\begin{equation}\label{tensor}
 L_x(v \otimes v') =L_x v\otimes w+ v \otimes L_x v', \quad i_x(v\otimes v')= i_x v \otimes v' +(-1)^{\vert v\vert} v \otimes i_x v'.  
\end{equation}
\end{remark}

\begin{theorem}\label{dif}
Let $V$ be a module in $\Mod(\uC_{\sbullet}(G))$ with structure given by
 $\rho\colon \uC_{\sbullet}(G) \rightarrow \End(V)$. There is a representation $\Dscr(\rho)\colon\TT\gfrak \rightarrow \End(V)$ given by
 \[\Dscr(\rho)(L_x)=\frac{d}{dt}\Big\rvert_{t=0} \rho(\exp(tx)),\]
and
\[ \Dscr(\rho)(i_x)=\frac{d}{dt}\Big\rvert_{t=0} \rho(\sigma[tx]).\]
Moreover, if $\phi\colon V \rightarrow W$ is a morphism in $\Mod(\uC_{\sbullet}(G))$, it is also a morphism in $\Rep(\TT\gfrak)$. One concludes that there is a functor
\[ \Dscr\colon \Mod(\uC_{\sbullet}(G)) \rightarrow \Rep(\TT\gfrak).\]
The functor $\Dscr$ is monoidal.
\end{theorem}

\begin{proof}
We write $\widehat{\rho}$ instead of $\Dscr(\rho)$. Clearly, $\widehat{\rho}(L_x)$ depends linearly on $x$. Let us show that so does $\widehat{\rho}(i_x)$. Consider the map
\[ \sigma\colon \gfrak \times I  \rightarrow G, \quad (x,s)\mapsto \exp(sx).\]
By hypothesis, we know that the map
\[ \rho_{\sigma} \colon \gfrak  \rightarrow \End(V),\quad x\mapsto \rho(\sigma[x])\]
is smooth. Its derivative $\rho_{\sigma}(0)\colon \gfrak \rightarrow \End(V)$ is a linear map and satisfies
\[ d\rho_{\sigma}(0)(x)=\widehat{\rho}(i_x).\]
We need to prove that the operators satisfy the Cartan relations:
\begin{eqnarray}
 {}[\widehat{\rho}(L_x),\widehat{\rho}(L_y)]&=&\widehat{\rho}(L_{[x,y]}) \\
 {}[\delta, \widehat{\rho}(i_x)] &=& \widehat{\rho}(L_x)\\
{} [\widehat{\rho}(i_x),\widehat{\rho}(i_y)]&=&0 \\
{}[ \widehat{\rho}(L_x),\widehat{\rho}(i_y)]&=&\widehat{\rho}(i_{[x,y]})
\end{eqnarray}
Here $\delta$ denotes the differential in $V$. For the first identity, we observe that $V$ is in particular a representation of $G$ and the formula for $\widehat{\rho}(L_x)$ gives the induced representation on the Lie algebra. Let us prove the second identity. Since $\rho\colon \uC_{\sbullet}(G) \rightarrow \End(V)$ is a morphism of complexes, we know that
\[ [\delta,\rho(\sigma[tx])]=\rho (\exp(tx))- \id .\]
Differentiating with respect to $t$ and evaluating at $t=0$ one obtains
\[ [\delta, \widehat{\rho}(i_x)]=\widehat{\rho}(L_x).\]
For the next identity it suffices to show that $[\widehat{\rho}(i_x), \widehat{\rho}(i_x)]=0$. Let us consider the quantity
\[ \frac{d}{dt}\Big\rvert_{t=0} \left(  \frac{d}{dt}\rho\left(\sigma[tx]^2\right) \right).\]
By Lemma \ref{square} we know that $\sigma[tx]^2=0$ so that the quantity above is zero.
On the other hand we compute
\begin{eqnarray*}
\frac{d}{dt}\Big\rvert_{t=0} \left(  \frac{d}{dt} \left( \rho(\sigma[tx]) \rho(\sigma[tx] \right) \right)&=&\frac{d}{dt}\Big\rvert_{t=0}\Bigg[\left(\frac{d}{dt}(\rho(\sigma[tx])\right) \rho(\sigma[tx])+ \rho(\sigma[tx]) \left(\frac{d}{dt}\rho(\sigma[tx])\right)\Bigg]\\
&=& \frac{d}{dt}\Big\rvert_{t=0}\left(\frac{d}{dt}\rho(\sigma[tx])\right) \rho(\sigma[0])+ \frac{d}{dt}\Big \rvert_{t=0}\rho(\sigma[tx]) \frac{d}{dt}\Big\rvert_{t=0} \rho(\sigma[tx])\\
&&+ \frac{d}{dt}\Big \rvert_{t=0}\rho(\sigma[tx]) \frac{d}{dt}\Big\rvert_{t=0} \rho(\sigma[tx])+ \rho(\sigma[0]) \frac{d}{dt}\Big\rvert_{t=0}\left(\frac{d}{dt}\rho(\sigma[tx])\right)\\\\
&=&\widehat{\rho}(i_x) \circ \widehat{\rho}(i_x) + \widehat{\rho}(i_x) \circ \widehat{\rho}(i_x)\\
&=&[\widehat{\rho}(i_x),\widehat{\rho}(i_x)]
\end{eqnarray*}
We conclude that the third identity holds.  It remains to prove the last relation. For this we consider the quantity
\[ A=\frac{d}{ds}\Big\rvert_{s=0}\frac{d}{dt}\Big\rvert_{t=0}\rho \big( \exp(tx) \sigma[sy] \exp(-tx) \big).\]
On the one hand, 
\begin{eqnarray*}
A&=&\frac{d}{ds}\Big\rvert_{s=0}\frac{d}{dt}\Big\rvert_{t=0}\rho(\exp(tx)) \rho( \sigma[sy]) \rho( \exp(-tx) )\\
&=&\frac{d}{ds}\Big\rvert_{s=0}\frac{d}{dt}\Big\rvert_{t=0}\exp(t \widehat{\rho}(L_x)) \rho( \sigma[sy])  \exp(-t \widehat{\rho}(L_x)) )\\
&=&\frac{d}{ds}\Big\rvert_{s=0}[\widehat{\rho}(L_x), \rho( \sigma[sy]) ]\\
&=&[\widehat{\rho}(L_x),\widehat{\rho}(i_y)].
\end{eqnarray*}
On the other hand
\begin{eqnarray*}
A&=&\frac{d}{ds}\Big\rvert_{s=0}\frac{d}{dt}\Big\rvert_{t=0}\rho \big( \exp(tx) \sigma[sy] \exp(-tx) \big)\\
&=&\frac{d}{ds}\Big\rvert_{s=0}\frac{d}{dt}\Big\rvert_{t=0}\rho \big( \sigma[s \Ad_{\exp(tx)}y] \big)\\
&=&\frac{d}{ds}\Big\rvert_{s=0} \rho(\sigma[s[x,y]])\\
&=&\widehat{\rho}(i_{[x,y]}).
\end{eqnarray*}
In order to prove that $\Dscr$ defines a functor it suffices to show that a map $\phi\colon V \rightarrow W$, which is $\uC_{\sbullet}(G)$-equivariant is also $\TT\gfrak$-equivariant. The fact that $\widehat{\rho}(L_x)$ commutes with $\varphi$ is standard. It remains to compute:
\[ \phi( \widehat{\rho}(i_x)(v))=\phi\left( \frac{d}{dt}\Big\rvert_{t=0} \rho\big(\sigma[tx](v)\right)= \frac{d}{dt}\Big \rvert_{t=0} \phi \big(\rho(\sigma[tx])(v)\big)=\frac{d}{dt}\Big \rvert_{t=0} \rho(\sigma[tx])(\phi (v))=\widehat{\rho}(i_x)(\phi (v)).\]
It only remains to prove that the functor $\Dscr$ is monoidal. Consider $V, V' \in \Mod$ with corresponding morphisms $\rho \colon \uC_{\sbullet}(G) \rightarrow \End(V)$ and $\rho' \colon \uC_{\sbullet}(G) \rightarrow \End(V)$. Let $\rho \otimes \rho'\colon \uC_{\sbullet}(G) \rightarrow \End(V\otimes V')$ be the representation on $V \otimes V'$ induced by the coproduct on $\uC_{\sbullet}(G)$. Let us compute:
\begin{eqnarray*}
\widehat{\rho \otimes \rho'}(L_x)(v \otimes v')&=&\frac{d}{dt}\Big\rvert_{t=0} (\rho \otimes \rho')(\exp(tx))(v\otimes v')\\
&=& \frac{d}{dt}\Big\rvert_{t=0} \rho(\exp(tx))(v)\otimes \rho'(\exp(tx)) (v')\\
&=& \frac{d}{dt}\Big\rvert_{t=0} \Big(\rho(\exp(tx))(v)\Big)\otimes v'+ v \otimes \frac{d}{dt}\Big\rvert_{t=0} \Big(\rho'(\exp(tx))(v')\Big)\\
&=&\widehat{\rho}(L_x)(v) \otimes v' + v \otimes \widehat{\rho}'(L_x) (v').
\end{eqnarray*}
Also,
\begin{eqnarray*}
\widehat{\rho \otimes \rho'}(i_x)(v \otimes v')&=&\frac{d}{dt}\Big\rvert_{t=0} (\rho \otimes \rho')(\sigma[tx])(v\otimes v')\\
&=& \frac{d}{dt}\Big\rvert_{t=0} \rho(\sigma[tx])(v)\otimes \rho'(\sigma[tx]) (v')\\
&=& \frac{d}{dt}\Big\rvert_{t=0} \Big(\rho(\sigma[tx])(v)\Big)\otimes v'+ (-1)^{\vert v \vert} v \otimes \frac{d}{dt}\Big\rvert_{t=0} \Big(\rho'(\sigma[tx])(v')\Big)\\
&=& \widehat{\rho}(i_x)(v) \otimes v' + (-1)^{\vert v \vert} v \otimes \widehat{\rho}'(i_x) (v').
\end{eqnarray*}
We conclude that $\widehat{\rho \otimes \rho'}$ is the tensor product of the representations $\widehat{\rho}$ and $\widehat{\rho}'$ as described by equation \eqref{tensor}.
\end{proof}

\begin{remark}
The functor $\Dscr$ is natural with respect to Lie group homomorphisms. That is, if $f\colon G \rightarrow H$ is a Lie group homomorphism then the following diagram commutes
\[ \xymatrix{
\Mod(\uC_{\sbullet}(H)) \ar[r]^-{(f_*)^*}\ar[d]_-{\Dscr}& \Mod(\uC_{\sbullet}(G))\ar[d]^-{\Dscr} \\
\Rep(\TT\mathfrak{h})\ar[r]^-{f^*} & \Rep(\TT\gfrak). 
}\]
\end{remark}


\subsection{The integration functor}
 Let us fix a simply connected Lie group $G$ with Lie algebra $\gfrak$, and a representation $\rho\colon \TT\gfrak \rightarrow \End(V)$. As before, given $x \in \gfrak$ we denote by $L_x \in \TT\gfrak^0$ and $i_x \in \TT\gfrak^{-1}$ the corresponding elements in $\TT\gfrak$. We also denote by $\rho\colon G \rightarrow \GL(V)$ the corresponding representation of the group.
 
  \begin{lemma}\label{com}
For each $ g \in G$, $\rho(g)$ commutes with the differential in $V$. That is, the map $\rho(g)$ is a morphism of cochain complexes.
 \end{lemma}
 
  \begin{proof}
 The statement clearly holds for $g=e$. Choose a path $\gamma\colon I \rightarrow G$ such that $\gamma(0)=e$ and 
 $\gamma(1)=g$. One computes,
$$
 \frac{d}{dt} [\delta,\rho(\gamma(t))] =\left[\partial,\frac{d}{dt} \rho(\gamma(t))\right]=[\delta,\rho(L_{dL_{\gamma(t)^{-1}}(\gamma'(t))})]= \rho(d(L_{dL_{\gamma(t)^{-1}}(\gamma'(t))}))=0. \qedhere
$$
 \end{proof}
 
 \begin{definition}
We will say that a differential form $\omega \in \Omega^{\sbullet}(G, \End(V))$ is $G$-equivariant if it satisfies
\[ L_g^* \omega=\rho(g)\circ \omega,\]
for all $g\in G$.
\end{definition}

 \begin{lemma}\label{commutes}
 An equivariant form is determined by its value at the identity as follows
 \begin{equation}\label{equivariant}
  \omega_g(y_1,\dots, y_k)= \rho(g) \circ \omega_e\left(x_1,\dots,x_n \right),
  \end{equation}
 where $y_i =(dL_g)_e (x_i)$. Moreover, the space of equivariant forms is closed under the De Rham differential $d$, and also with respect to the differential $\delta$.
 \end{lemma}
 
 \begin{proof}
 Suppose that $\omega$ vanishes at the identity. Taking $y_i= (dL_g)_e (x_i)$ one computes: 
 \[0=\rho(g)\circ \eta_e(x_1,\dots,x_k)= (L_g^*\omega)_e(x_1, \dots, x_k)= \omega_g(y_1, \dots , y_n).\]
Since the $y_i$ are arbitrary, on concludes that $\omega$ vanishes. Conversely, it is clear that once $\omega$ is specified at the identity, formula \eqref{equivariant} provides an equivariant extension. In order to prove that equivariant forms are preserved by the De Rham differential we compute
 \[L_g^*(d\omega)= d( L_g^* \omega)=d( L_{\rho(g)}\circ \omega )=L_{\rho(g)} \circ d\omega .\]
 For the last statement, we compute
 \[  L_g^*(\delta \omega)= \delta( L_g^* \eta)=\delta ( \rho(g) \circ \omega )= \rho(g) \circ \delta \omega ,\]
 where we have used Lemma \ref{com} in the last step.
 \end{proof}
 
 \begin{lemma}\label{adjoint}
Let $\rho\colon\TT\gfrak \to \End(V)$ be a representation. Then:
\[ \rho\left(i_{\Ad_g x}\right)= \rho(g) \circ \rho(i_x) \circ \rho(g^{-1}).\]
\end{lemma}

\begin{proof}
We need to show that the map $\rho\colon \TT\gfrak^{-1} \rightarrow \End(V)$ is $G$-equivariant. Since $G$ is connected, it suffices to show that it is $\gfrak$-equivariant.  For this we compute:
\[ \rho(L_y i_x)=\rho(i_{[y,x]})=\rho([L_y,i_x])=[\rho(L_y), \rho(i_x)]=L_y\rho(i_x). \qedhere\]
\end{proof}

\begin{definition}
 For each $k \geq 0$, we denote by $\Phi_V^{(k)} \in \Omega^k(G, \End^{-k}(V))$ the unique $G$-equivariant form such that
 \[ \Phi_V^{(k)}(e)(x_1, \dots, x_k)= \rho(i_{x_1}) \circ \cdots \circ \rho(i_{x_k}).\]
Since $\rho$ is a representation of $\TT\gfrak$, this expression is skew-symmetric.
\end{definition}

 \begin{lemma}\label{derivative}
 Let $\omega$ be an $G$-equivariant form and $y, x_1, \dots , x_k$ be left invariant vector fields on $G$. Then
 \[ (L_y\omega)_e(x_1, \dots x_k)=  \rho(L_y) \circ \omega_e (x_1, \dots x_k).\]
 \end{lemma}
 
  \begin{proof}
 We compute
 \begin{eqnarray*}
 (L_y\omega)_e(x_1, \dots x_k) &=&\frac{d}{dt}\Big \rvert_{t=0} \omega_{\exp(ty)}(x_1, \dots x_k)\\
 &=&\frac{d}{dt}\Big \rvert_{t=0} (L_{\exp(ty)}^*\omega)_e(x_1, \dots x_k)\\
 &=&\frac{d}{dt}\Big \rvert_{t=0}\rho(\exp(ty))\circ \omega_e(x_1, \dots x_k) \\
 &=& \rho(L_y) \circ \omega_e(x_1, \dots x_k) .  \qedhere
 \end{eqnarray*}
 \end{proof}
 
  \begin{lemma}\label{tele}
 For each $k \geq 1$ and left invariant vector fields $y, x_1,\dots, x_k$ on $G$, the following equation holds:
 \[ \sum_{j=1}^k (-1)^{j+1}\Phi_V^{(k)}(e)([y, x_j], x_1, \dots, \widehat{x}_j, \dots , x_k)=\rho(L_y) \circ \rho(i_{x_1})\circ\cdots \circ\rho(i_{x_k})-\rho(i_{x_1}) \circ \cdots \circ \rho(i_{x_k}) \circ \rho(L_y). \]
 \end{lemma}
 
 \begin{proof}
 We argue by induction on $k$. For $k=1$ we have:
 \[ \Phi_V^{(1)}(e)([y,x])=\rho(i_{[y,x]})=\rho(L_{y}) \circ \rho(i_x) -\rho(i_x) \circ \rho(L_{y}).\]
 Let us now assume that it holds for $k-1$ and compute:
 \begin{eqnarray*}
&& \!\!\!\!\!\!\!\!\!\!\!\!\!\!\!\!\!\! \sum_{j=1}^k (-1)^{j+1}\Phi_V^{(k)}(e)([y, x_j], x_1, \dots, \widehat{x}_j, \dots , x_k)\\  
 &=&\sum_{j=1}^{k-1}(-1)^{j+1} \Phi_V^{(k)}(e)([y, x_j], x_1, \dots, \widehat{x}_j, \dots , x_k) 
 + (-1)^{k+1} \Phi_V^{(k)}(e)([y, x_k], x_1,\dots , x_{k-1}) \\
 &=&\sum_{j=1}^{k-1}(-1)^{j+1} \Phi_V^{(k)}(e)([z, x_j], x_1, \dots, \widehat{x}_j, \dots , x_{k-1})\circ \rho(i_{x_k}) 
 + \Phi_V^{(k)}(e)( x_1,\dots , x_{k-1},[y,x_k]) \\
  &=& \rho(L_y) \circ \rho(i_{x_1}) \circ \cdots \circ \rho(i_{x_k})- \rho(i_{x_1}) \circ \cdots \circ \rho(i_{x_{k-1}}) \circ \rho(L_y) \circ \rho(i_{x_k})  \\
  &&+\rho(i_{x_1}) \circ \cdots \circ \rho(i_{x_{k-1}}) \circ \rho(L_y)\circ \rho(i_{x_k}) - \rho(i_{x_1}) \circ \cdots \circ \rho(i_{x_k}) \circ \rho(L_{y})\\
  &=&  \rho(L_y)\circ \rho(i_{x_1}) \circ \cdots \circ \rho(i_{x_k})- \rho(i_{x_1}) \circ \cdots \circ \rho(i_{x_k}) \circ \rho(L_{y}). \qedhere
 \end{eqnarray*}
 \end{proof}
 
 \begin{lemma}\label{total}
 The differential forms $\Phi_V^{(k)}$ satisfy
 \[ d \Phi_V^{(k)}=(-1)^{k} \delta \Phi_V^{(k+1)}.\]
 \end{lemma}
 
  \begin{proof}
 Since both sides of the equation are equivariant, it suffices to show that they coincide at the identity.
 Let us fix left invariant vector fields $x_1, \dots x_{k+1}$ and compute the right hand side
 \begin{eqnarray*}
  (\delta\Phi_V^{(k+1)})_e(x_1,\dots x_{k+1})=\delta \left( \rho(i_{x_1}) \circ \cdots \circ \rho(i_{x_{k+1}}) \right)   =\sum_{i=1}^{k+1} (-1)^{i+1}  \rho(i_{x_1}) \circ \cdots \circ \rho(L_{x_i}) \circ \cdots  \circ \rho(i_{x_{k+1}}).
 \end{eqnarray*}
 For the left hand side we use Lemma \ref{derivative} and Lemma \ref{tele} to compute
  \begin{eqnarray*}
 && \!\!\!\! \!\!\!\! \!\!\!\! \!\!\!\! (d \Phi_V^{(k)})_e(x_1, \dots ,x_{k+1}) \\
 &=&  \sum_{i<j} (-1)^{i+j+k} \Phi_V^{(k)}(e)( [x_i, x_j], \dots , \widehat{x}_i, \dots , \widehat{x}_j, \dots , x_{k+1}\\
 && + \sum_i (-1)^{i+k+1}(L_{x_i} \Phi_V^{(k)})_e(x_1, \dots , \widehat{x}_i, \dots, x_{k+1})\\
&=&  \sum_{j>1}\sum_{i<j} (-1)^{i+j+k} \Phi_V^{(j-1)}(e)( [x_i, x_j], \dots , \widehat{x}_i, \dots , x_{j-1}) \circ \rho(i_{x_{j+1}}) \circ \cdots \circ \rho(i_{x_{k+1}})\\
&& + \sum_i (-1)^{i+k+1}\rho(L_{x_i}) \circ \rho(i_{x_1}) \circ \cdots \circ \widehat{\rho(i_{x_i})} \circ \cdots \circ \rho(i_{x_{k+1}})\\
&=&  \sum_{j>1}(-1)^{j+k}\sum_{i<j} (-1)^{i+1} \Phi_V^{(j-1)}(e)( [x_i, x_j], \dots , \widehat{x}_i, \dots , x_{j-1}) \circ \rho(i_{x_{j+1}}) \circ \cdots \circ \rho(i_{x_{k+1}})\\
&&  + \sum_i (-1)^{i+k+1}\rho(L_{x_i}) \circ \rho(i_{x_1}) \circ \cdots \circ \widehat{\rho(i_{x_i})} \circ \cdots \circ \rho(i_{x_{k+1}})\\
&=&  \sum_{j>1}(-1)^{j+k}\Big[\rho(L_{x_j}) \circ \rho(i_{x_1})\circ \cdots \widehat{\rho(i_{x_j})} \circ \cdots \circ  \rho(i_{x_{k+1}}) -\rho(i_{x_1}) \circ \cdots \circ \rho(L_{x_j}) \circ \cdots \circ \rho(i_{x_{k+1}}) \Big]\\
&& + \sum_j (-1)^{j+k+1}\rho(L_{x_j}) \circ \rho(i_{x_1}) \circ \cdots \circ \widehat{\rho(i_{x_j})} \circ \cdots \circ \rho(i_{x_{k+1}})\\
&=& \sum_{j}(-1)^{j+k+1} \rho(i_{x_1}) \circ \cdots \circ \rho(L_{x_j}) \circ \cdots \circ \rho(i_{x_{k+1}}) \\
&=& (-1)^k (\delta \Phi_V^{(k+1)})_e(x_1,\dots x_{k+1}). \qedhere
 \end{eqnarray*}
 \end{proof}
 
 \begin{lemma}\label{product}
Let $\mu\colon G \times G \rightarrow G$ be the multiplication map and $\pi_1,\pi_2 \colon G \times G \to G$ the projections.
The form $\Phi_V^{(k)}$ satisfies the equation
\[\mu^*\Phi_V^{(k)}= \sum_{i+j=k}(-1)^{ij}\pi^*_1\Phi_V^{(i)} \wedge \pi^*_2\Phi_V^{(j)}.\]
\end{lemma}

\begin{proof}
Let us fix $g, h \in G$ and vectors $y_1, \dots, y_i \in T_g G$, $y_{i+1}, \dots , y_{k} \in T_h G$. Denote by $x_l$ 
the vectors in $\gfrak$ characterized by $ (dL_{g})_e(x_l)=y_l$ for $l\leq i$ and $(dL_h)_e(x_l)=y_l$ for $l >i$.
We use Lemma \ref{adjoint} to compute the left hand side:
\begin{align*}
&(\mu^*\Phi_V^{(k)})_{(g,h)}(y_1, \dots ,y_{k}) \\
&\quad=\Phi_V^{(k)}(gh)((dR_h)_g (y_1), \dots , (dR_h)_g (y_i), (dL_g)_h (y_{i+1}), \dots, (dL_g)_h (y_{k}) )\\
&\quad=\rho(gh) \circ \Phi_V^{(k)}(e)(\Ad_{h^{-1}}x_1, \dots, \Ad_{h^{-1}}x_i,x_{i+1}, \dots,x_{k}) \\
&\quad= \rho(g) \circ \rho(h) \circ \rho(i_{\Ad_{h^{-1}}x_1}) \circ \cdots \circ \rho(i_{\Ad_{h^{-1}}x_i}) \circ \rho(i_{x_{i+1}}) \circ \cdots \circ \rho(i_{x_k})\\
&\quad= \rho(g)\circ \rho(h) \circ \rho(h^{-1})\circ \rho(i_{x_1}) \circ \rho(h) \circ \cdots \circ \rho(h^{-1}) \circ \rho(i_{x_i}) \circ \rho(h) \circ \rho(i_{x_{i+1}})  \circ \cdots \circ \rho(i_{x_k})\\
&\quad= \Phi_V^{(i)}(g)( y_1, \dots,y_i) \circ \Phi_V^{(k-i)}(h)(y_{i+1}, \dots ,  y_k)
\end{align*}
On the other hand, the right hand side is:
\begin{align*}
\sum_{j+l=k}(-1)^{jl}\big(\pi^*_1\Phi_V^{(j)} \wedge \pi^*_2\Phi_V^{(l)}\big)_{(g,h)}(y_1,\dots, y_{k})&=(-1)^{i(k-i)}\big(\pi^*_1\Phi_V^{(j)} \wedge \pi^*_2\Phi_V^{(l)}\big)_{(g,h)}(y_1,\dots, y_{k})\\
&= \Phi_V^{(i)}(g)( y_1, \dots,y_i) \circ \Phi_V^{(k-i)}(h)(y_{i+1}, \dots ,  y_k). &\qedhere
\end{align*}
\end{proof}

\begin{definition}
Let $G$ be a Lie group and $(V,\delta)$ a finite dimensional cochain complex. A representation form is a differential form $\Phi \in \Omega^{\sbullet}(G, \End(V))$ such that
\begin{itemize}
\item  $\Phi=\sum_{k \geq 0} \Phi^{(k)}$, where $\Phi^{(k)} \in \Omega^k(G, \End^{-k}(V))$.
\item $\Phi^{(0)}(e)=\id_V$.
\item $d\Phi^{(k)} =(-1)^k \delta \Phi^{(k+1)}$.
\item 
$\mu^*\Phi^{(k)}=\sum_{i+j=k}(-1)^{ij}\pi^*_1\Phi^{(i)}\wedge \pi_2^* \Phi^{(j)}$.
\end{itemize}
\end{definition}

\begin{lemma}
Let $\Phi$ be a representation form. Then $\Phi^{(0)}$ is a representation of $G$ and $\Phi$ is $G$-equivariant with respect to this action.
\end{lemma}

\begin{proof}
Let us show that $\Phi^{(0)} \colon G \rightarrow \End(V)$ is a representation. By hypothesis, $\Phi^{(0)}(e)=\id_V$. Also
\[ \Phi^{(0)}(gh)=\mu^*\Phi^{(0)}(g,h)=\Phi^{(0)}(g) \circ \Phi^{(0)}(h).\]
Let us now show that the differential form $\Phi^{(k)}$ is $G$-equivariant with respect to this action. Take $y_1, \dots, y_k \in T_hG$ and $x_i \in \gfrak$ with $y_i=(dL_h)_e(x_i)$. Then
\[(\mu^*\Phi^{(k)})_{(g,h)}(y_1, \dots , y_k)=\Phi^{(k)}(gh)((dL_g)_h (y_1), \dots, (dL_g)_h (y_k))=(L_g^* \Phi^{(k)})_h(y_1, \dots, y_k).\]
On the other hand
\begin{eqnarray*}
(\mu^*\Phi^{(k)})_{(g,h)}(y_1, \dots , y_k)&=&\sum_{i+j=k}(-1)^{ij} \big(\pi^*_1\Phi^{(i)} \wedge \pi^*_2\Phi^{(j)} \big)_{(g,h)}(y_1,\dots, y_k)\\
&=& \Phi^{(0)}(g) \circ \Phi^{(k)}(h) (y_1, \dots, y_k). \qedhere
\end{eqnarray*}
\end{proof}

\begin{lemma}\label{muk}
Let $\Phi \in \Omega^{\sbullet}(G, \End(V)$ be a representation form and denote by $\mu_p\colon G^p \rightarrow G$ the multiplication map. Then:
\[ \mu_p^*\Phi^{(k)}=\sum_{j_1+\dots +j_p=k}(-1)^{\sum_{l=1}^p j_l(j_{l-1}+ \dots +j_1)}\pi^*_1\Phi^{(j_1)}\wedge \dots \wedge \pi^*_p\Phi^{(j_p)}.\]
\end{lemma}

\begin{proof}
We argue by induction on $p$. The case $p=1$ is trivial and $p=2$ holds by definition. Let us assume that the statement holds for $p-1$. We use the notation
\[ \widetilde{\mu}_{p-1}\colon G^p \rightarrow G,\quad  \widetilde{\mu}_{p-1}(g_1, \dots , g_p)=g_2 \dots g_p\] and compute
\begin{eqnarray*}
\mu_p^*\Phi^{(k)}&=& (\mu \circ (\id \times \mu_{p-1}))^*\Phi^{(k)}=(\id \times \mu_{p-1})^* \circ \mu^*\Phi^{(k)}\\
&=&(\id \times \mu_{p-1})^*\left( \sum_{i+j=k}(-1)^{ij}\pi^*_1\Phi^{(i)}\wedge  \pi^*_2\Phi^{(j)}\right)\\
&=& \sum_{i+j=k}(-1)^{ij}\pi^*_1\Phi^{(i)}\wedge  \widetilde{\mu}_{p-1}^*\Phi^{(j)}\\
&=&\sum_{i+j=k}(-1)^{ij}\pi^*_1\Phi^{(i)} \wedge \left(\sum_{j_1+\dots +j_{p-1}=j}(-1)^{\sum_{l=1}^{p-1} j_l(j_{l-1}+ \dots +j_1)}\pi^*_2\Phi^{(j_1)}\wedge \dots \wedge \pi^*_{p}\Phi^{(j_{p-1})}\right)\\
&=&\sum_{j_1+\dots +j_p=k}(-1)^{\sum_{l=1}^p j_l(j_{l-1}+ \dots +j_1)}\pi^*_1\Phi^{(j_1)}\wedge \dots \wedge \pi^*_p\Phi^{(j_p)} \qedhere
\end{eqnarray*}
\end{proof}

\begin{lemma}\label{injective}
Let $\Phi, \Psi \in \Omega^{\sbullet}(G, \End(V))$ be representation forms which coincide in degrees $0$ and $1$. Then $\Phi= \Psi$.
\end{lemma}

\begin{proof}
For each $p\geq 1$ denote by  $\mu_p\colon G^p \rightarrow G $ the multiplication map. Consider the point $(e , \dots, e ) \in G^p$ and vectors $x_i \in \gfrak$ seen as a tangent vector in the $i$th copy of $G$. We use Lemma \ref{muk} to compute:
\begin{eqnarray*}
\Phi(e)(x_1, \dots, x_k)&=&(\mu_p^*\Phi)_{(e,\dots,e)}(x_1, \dots, x_k)\\
&=&(-1)^{1+2+\dots +k-1}(\pi^*_1 \Phi^{(1)} \wedge \dots \wedge \pi^*_k \Phi^{(1)})_{(e,\dots,e)} (x_1, \dots , x_k).
\end{eqnarray*}
This expression depends only on $\Phi^{(1)}$. By symmetry, the corresponding expression for $\Psi$ depends only on $\Psi^{(1)}$. We conclude that $\Phi^{(k)}(e)=\Psi^{(k)}(e)$. Since both forms are $G$-equivariant, they are equal.
\end{proof}

\begin{proposition}\label{propform}
Let $G$ be a simply connected Lie group and $\rho\colon \TT\gfrak \rightarrow \End(V)$ a representation.
There is a unique representation form $\Phi_{\rho}$ such that $(d\Phi_{\rho}^{(0)})_e(x)=\rho(L_x)$ and $\Phi_{\rho}^{(1)}(e)(x)=\rho(i_x)$. Moreover, this construction is a bijective correspondence between representations of $\TT\gfrak$ and representation forms.
\end{proposition}

\begin{proof}
As before, one defines $\Phi^{(k)}_{\rho}$ to be the unique equivariant form such that
\[ \Phi^{(k)}_{\rho}(x_1,\dots,x_k)=\rho(i_{x_1}) \circ \dots \circ \rho(i_{x_k}).\]
 By, Lemma \ref{total} and Lemma \ref{product} we know that $\Phi_{\rho}$ is a representation form. It is clear that the correspondence is injective. It remains to show that the map is surjective. Let $\Phi$ be a representation form. We can define a representation $\rho_{\Phi}$ of $\TT\gfrak$ by
\[ \rho_{\Phi}(L_x)= (d\Phi^{(0)})_e(x), \quad  \rho_{\Phi}(i_x)= \Phi^{(1)}(e)(x).\]
Clearly, $\rho_\Phi([L_x,L_y])=[\rho_\Phi(L_x), \rho_\Phi(L_y)]$. To show that $\delta(\rho_\Phi(i_x))= \rho_\Phi(L_x)$ we compute:
\[\delta(\rho_\Phi(i_x)) =(\delta \Phi^{(1)})_e(x)=(d\Phi^{(0)})_e(x)=\rho_\Phi(L_x).\]
 We fix two vectors $x, y\in \gfrak$ seen as tangent vectors to the first and second copy of $G$ in $G\times G$. Then we compute
\begin{align*}\label{2}
\Phi^{(2)}(e)(x,y) &=\big(\mu^*\Phi^{(2)}\big)_{(e,e)}(x,y)=-\big(\pi^*_1\Phi^{(1)}\wedge \pi^*_2\Phi^{(1)}\big)_{(e,e)}(x, y) \\
&=\Phi^{(1)}(e)(y) \circ \Phi^{(1)}(e)(x)= \rho_\Phi(i_y) \circ \rho_\Phi(i_x) .
\end{align*}
Since the expression is skew-symmetric, one concludes that $[\rho_\Phi(i_x), \rho_\Phi(i_y)]=0$.
It only remains to show that $\rho_\Phi(i_{[x,y]})=[\rho_\Phi(i_x), \rho_\Phi(L_y)]$. We consider left invariant vector fields $x$ and $y$ on $G$ and use Lemma \ref{derivative} to compute
\begin{eqnarray*}  (d\Phi^{(1)})_e(x,y)&=&\Phi^{(1)}(e)([x,y])-\big(L_{x}\Phi^{(1)}\big)_e(y)+ \big(L_{y}\Phi^{(1)}\big)_e(x)\\
&=&\Phi^{(1)}(e)([x,y]) - \rho_\Phi(L_x)\circ  \rho_\Phi(i_y)+\rho_\Phi(L_y) \circ \rho_\Phi(i_x). 
\end{eqnarray*}
On the other hand,
\begin{eqnarray*}   (d\Phi^{(1)})_e(x,y)&=& -(\delta\Phi^{(2)})_{e}(x,y)\\
&=&- \delta (\Phi^{(1)}(e)(x) \circ \Phi^{(1)}(e)(y))\\
&=&-\delta( \rho_\Phi(i_x) \circ \rho_\Phi(i_y))\\
&=& - \rho_\Phi(L_x) \circ \rho_\Phi(i_y) +\rho_\Phi(i_x) \circ \rho_\Phi(L_y). 
\end{eqnarray*}
We conclude that $\rho_\Phi(i_{[v,w]})=[\rho_\Phi(i_x), \rho_\Phi(L_y)]$. Lemma \ref{injective} implies that the representation form associated to $\rho_\Phi$ is $\Phi$.
\end{proof}

\begin{definition}
Given a differential form $\omega = T \otimes \eta  \in \Omega^{\sbullet}(G,\End(V))$ and a simplex $\sigma \in \uC_{\sbullet}(G)$ we define the integral of $\omega$ over $\sigma$ by
\[ \int_{\Delta_k} \sigma^*\omega= T \int_{\Delta_k} \sigma^*\eta.\]
\end{definition} 

\begin{proposition}\label{mult}
Let $\Phi \in \Omega^{\sbullet}(G, \End(V))$ be a differential form with values in $\End(V)$.
The map
\[ \rho: \uC_{\sbullet}(G) \rightarrow \End(V), \quad
 \rho(\sigma)= \int_{\Delta_k} \sigma^*\Phi\]
is a morphism of differential graded algebras if and only if $\Phi$ is a representation form.
\end{proposition}

\begin{proof}
The map $\rho$ is a degree zero map if and only if $\Phi= \sum_{k \geq 0} \Phi^{(k)}$, where
$\Phi^{(k)} \in \Omega^k(G, \End^{-k}(V))$. Also, $\rho(1)=1$ if and only if $\Phi^{(0)}(e)=\id_V$. Let us prove that
$\rho$ preserves the product if and only if
\begin{equation}\label{ij}
 \mu^*\Phi^{(k)}=\sum_{i+j=k} (-1)^{ij} \pi^*_1\Phi^{(i)} \wedge \pi^*_2\Phi^{(j)}. 
 \end{equation}
 Suppose equation \eqref{ij} is satisfied for all $k$. Then, one computes
 \begin{align*}
\rho(\mu \circ (\sigma \times \nu)\circ \overline{\chi})&= \int_{\Delta_{r+s}}(\mu \circ (\sigma \times \nu)\circ \overline{\chi})^*\Phi\\
&=\int_{\Delta_{r+s}}\overline{\chi}^* (\sigma \times \nu)^* \mu^* \Phi\\
&=\sum_{i+j=r+s}(-1)^{ij}\int_{\Delta_{r+s}}\overline{\chi}^*   (\sigma \times \nu)^* \big(\pi_1^*\Phi^{(i)} \wedge \pi_2^*\Phi^{(j)}\big)\\
&=\sum_{i+j=r+s}(-1)^{ij}\int_{\Delta_{r+s}}\overline{\chi}^*\big( \pi_1^*\sigma^*\Phi^{(i)} \wedge \pi_2^*\nu^*\Phi^{(j)}\big)\\
&=(-1)^{rs}\int_{\Delta_{r+s}}\overline{\chi}^*\big( \pi_1^*\sigma^* \Phi^{(r)}  \wedge \pi_2^*\nu^*\Phi^{(s)}\big)\\
&=(-1)^{rs}\int_{\overline{\chi}(\Delta_{r+s})} \pi_1^*\sigma^* \Phi^{(r)}  \wedge \pi_2^*\nu^*\Phi^{(s)}.
\end{align*}
Therefore
\begin{align*}
\rho( \sigma \abxcup  \nu)&= \sum_{\chi \in \mathfrak{S}_{r,s}} (-1)^{\vert\chi\vert+rs} \int_{\overline{\chi}(\Delta_{r+s})} \pi_1^*\sigma^* \Phi^{(r)}  \wedge \pi_2^*\nu^*\Phi^{(s)}\\
&=(-1)^{rs}\int_{\Delta_r \times \Delta_s} \pi_1^*\sigma^* \Phi^{(r)}  \wedge \pi_2^*\nu^*\Phi^{(s)}\\
&=\left(\int_{\Delta_r }\sigma^* \Phi^{(r)}\right)\circ \left(\int_{\Delta_s } \nu^* \Phi^{(s)}\right)\\
&=\rho(\sigma) \circ \rho(\nu).
\end{align*}
Conversely, suppose that $\rho$ preserves the product. Then
\begin{align*}
\rho( \sigma  \abxcup \nu)&= \sum_{\chi \in \mathfrak{S}_{r,s}}  \int_{\overline{\chi}(\Delta_{r+s})} (\sigma \times \nu)^*  \mu^*\Phi^{(r+s)}\\
&=\int_{\Delta_r \times \Delta_s} (\sigma \times \nu)^* \mu^*\Phi^{(r+s)}
\end{align*}
is equal to
\[
\rho(\sigma) \circ \rho(\nu)=\left(\int_{\Delta_r }\sigma^*\Phi^{(r)}\right)\left(\int_{\Delta_s } \nu^*\Phi^{(s)}\right)=(-1)^{rs}\int_{\Delta_r \times \Delta_s}  (\sigma \times \nu)^* \big(\pi^*_1 \Phi^{(r)} \wedge \pi^*_2\Phi^{(s)}\big).
\]
Since $\sigma$ and $\nu$ are arbitrary simplices, we conclude that Equation \eqref{ij} holds.

It remains to show that $\rho$ is a morphism of complexes if and only if 
\begin{equation}\label{d}
d\Phi^{(k)}=(-1)^k \delta \Phi^{(k+1)}.
\end{equation}
Let us suppose that Equation \eqref{d} holds and use Stokes' theorem to compute
\begin{eqnarray*}
\rho(\partial \sigma)&=& \int_{\partial \Delta_k} \sigma^*\Phi^{(k-1)}\\
&=&(-1)^{k-1}\int_{\Delta_k} \sigma^*(d \Phi^{(k-1)})\\
&=&\int_{ \Delta_k} \sigma^*(\delta \Phi^{(k)})\\
&=&\delta \left(\int_{ \Delta_k} \sigma^* \Phi^{(k)}\right)\\
&=& \delta (\rho(\sigma)).
\end{eqnarray*}
Conversely, if $\rho$ is a morphism of complexes then
\[ \int_{ \Delta_k} (-1)^{k-1}\sigma^*(d \Phi^{(k-1)})=\int_{ \Delta_k} \sigma^*(\delta \Phi^{(k)}).\]
Since the simplex $\sigma$ is arbitrary, this implies Equation \eqref{d}.
\end{proof}

\begin{theorem}\label{int}
Let $G$ be a simply connected Lie group. There is a functor:
\[ \Iscr\colon\Rep(\TT\gfrak) \rightarrow \Mod(\uC_{\sbullet}(G)) \]
defined as follows. Fix $V \in \Rep(\TT\gfrak)$ with associated homomorphism $\rho\colon \TT\gfrak \to \End(V)$ and corresponding representation form $\Phi_V$. As a complex of vector spaces, $\Iscr(V)$ is equal to $V$. The module structure is given by the map
\[ \Iscr(\rho)\colon \uC_{\sbullet}(G)\rightarrow \End(V), \quad \sigma \mapsto \int_{\Delta_k} \sigma^*\Phi_V.\]
If $\phi\colon V \to V'$ is a morphism in $\Rep(\TT\gfrak)$, then the same underlying linear map is also a morphism in $\Mod(\uC_{\sbullet}(G))$.
\end{theorem}

\begin{proof}
Proposition \ref{mult} implies that $\Iscr(\rho)$ is a morphism of algebras. Since the representation $\Iscr(\rho)$ is given by integration of a differential form, it is an object in $\Mod(\uC_{\sbullet}(G))$.
It only remains to show that a morphism $\phi\colon V \rightarrow V'$ in $\Rep(\TT\gfrak)$ is also a morphism in $\Mod(\uC_{\sbullet}(G))$. Let $\rho$ and $\rho'$ be the homomorphisms giving $V$ and $V'$ the structure of representations of $\TT\gfrak$. Consider also $\Phi_{V} \in \Omega(G, \End(V))$ and $\Phi_{V'}\in \Omega(G, \End(W))$ the corresponding representation forms. We want to show that for all $\sigma \in \uC_{\sbullet}(G)$ and all $v \in V$
\[ \phi \big(\Iscr(\rho)(\sigma)(v)\big)=\Iscr(\rho')(\sigma)( \phi(v)).\]
By the usual Lie theory, the statement is true in case $\sigma$ is a zero simplex.
We claim that
\begin{equation}\label{betas}
\phi \Big(\Phi_V^{(k)}(g)(y_1, \dots,y_k)(v)\Big)=\Phi_{V'}^{(k)}(g)(y_1, \dots, y_k)(\phi(v)). 
\end{equation}
In order to prove this we may assume that $y_i = (dL_g)_e(x_i)$ and compute
\begin{align*}
\phi \Big(\Phi_V^{(k)}(g)(y_1, \dots,y_k)(v)\Big)&= \phi \Big(\rho(g) \circ \Phi_V^{(k)}(e)(x_1, \dots,x_k) (v)\Big)\\
&= \phi \Big(\rho(g) \circ \rho(\iota_{x_1}) \circ \dots \circ \rho(\iota_{x_k}) (v)\Big)\\
&= \rho'(g) \circ \rho'(\iota_{x_1}) \circ \dots \circ \rho'(\iota_{x_k}) (\varphi(v))\\
&=\Phi_{V'}^{(k)}(g)(y_1, \dots, y_k)(\phi(v)).
\end{align*}
We conclude that Equation \eqref{betas} holds. Using this we compute
\begin{align*}
 \phi \big(\Iscr(\rho)(\sigma)(v)\big)&= \phi \left( \int_{\Delta_k} \sigma^*\Phi_V(v)\right)= \int_{\Delta_k} \phi \left(\sigma^*\Phi_V(v)\right)=\int_{\Delta_k}\sigma^*(\phi \circ \Phi_V)(v) \\
 &=\int_{\Delta_k}\sigma^*(\Phi_{V'} \circ \phi(v))=\Iscr(\rho')(\sigma)(\phi(v)). & \qedhere
\end{align*}
\end{proof}

\begin{remark}
The functor $\Iscr$ is natural with respect to Lie group homomorphisms. That is, if $f\colon G \rightarrow H$ is a Lie group homomorphism then the following diagram commutes
\[ \xymatrix{
\Mod(\uC_{\sbullet}(H)) \ar[r]^-{(f_*)^*}& \Mod(\uC_{\sbullet}(G)) \\
\Rep(\TT\mathfrak{h})\ar[r]^-{f^*} \ar[u]^-{\Iscr} & \Rep(\TT\gfrak). \ar[u]_-{\Iscr}
}\]
\end{remark}

\begin{example}
Suppose that $V$ is a representation of  $\TT\gfrak$ with corresponding representation form $\Phi_V$. We will fix vectors $x_1, \dots ,x_k \in \gfrak$ and let $\sigma\colon \Delta_k \rightarrow G$ be the simplex \[\sigma(t_1, \dots, t_k)=e^{t_1 x_1} \cdots e^{t_kx_k}.\]
We also set
\[ A_i= \rho (L_{x_i}) ,\quad B_i = \rho(i_{x_i}).\]
A simple computation shows that
\[ \sigma^*\Phi_V^{(k)}=B_1 e^{t_1 A_1} \cdots B_k e^{t_k A_k} dt_1 \cdots dt_k.\]
It follows that
\[\Iscr(\rho)(\sigma)= \int_{\Delta_k}\sigma^*\Phi_V^{(k)}= \int_{\Delta_k} B_1 e^{t_1 A_1} \cdots B_k e^{t_k A_k} dt_1 \dots dt_k.\]
Let us compute $\rho(\sigma)$ explicitly in the case $k=1$
\begin{eqnarray*}
\Iscr(\rho)(\sigma)= \int_0^1 B e^{tA}dt=B \sum_{k\geq 0} \frac{A^k }{(k+1)!}=B\Big(\frac{e^A-1}{A}\Big).
\end{eqnarray*}
A slightly longer computation shows that in general
\[ \Iscr(\rho)(\sigma)= \sum_{j_1, \dots , j_k \geq 0} \frac{B_1 A^{j_1}_1 \cdots B_kA_k^{j_k}}{j_1! \cdots j_k!(j_k+1) (j_k +j_{k-1}+2)\cdots (j_k+\cdots + j_1+k)}.\]
\end{example}


\subsection{$\Rep(\TT\gfrak)$ is equivalent to $\Mod(\uC_{\bullet}(G))$}
Here we prove our main result, Theorem \ref{main}, which states that the functors $I$ and $\zeta$ are inverses to one another.

\begin{lemma}\label{b1}
Let $\rho\colon \TT\gfrak \rightarrow \End(V)$ be a representation with corresponding representation form $\Phi_V$. Given $x \in \gfrak$  we denote by $\sigma[x]$ the $1$-simplex
$\sigma[x](t)=\exp(tx)$. Then
\[ \sigma[x]^*\Phi_V^{(1)}= \exp(t\rho(L_x)) \circ\rho(i_x)dt.\]
\end{lemma}

\begin{proof}
One computes
\[\big(\sigma[x]^*\Phi_V^{(1)}\big)_t \left(\frac{\partial}{\partial t}\right)= \Phi_V^{(1)} (\exp(tx))((dL_{\exp(tx)})_e (x))= \rho(\exp(tx))\circ \rho(i_x)=\exp(t\rho(L_x)) \circ \rho(\iota_x). \qedhere\]
\end{proof}

\begin{lemma}\label{left}
The functor $\Dscr$ is left inverse to $\Iscr$. Indeed,
\[ \Dscr \circ \Iscr= \id_{\Rep(\TT\gfrak)}.\]
\end{lemma}

\begin{proof}
Take $V \in \Rep(\TT\gfrak)$ with corresponding representation form $\Phi_V$. Let $\rho,\rho': \TT\gfrak \rightarrow \End(V)$ be the maps that give $V$ and $ \Dscr(\Iscr(V))$ the structure of representations. We need to prove that $ \rho= \rho'$. The fact that $\rho(L_x)=\rho'(L_x)$ is standard. One can use Lemma \ref{b1} to compute $\rho'(\iota_x)$ as follows
\begin{align*}
\rho'(i_x)&=\frac{d}{dt}\Big \vert_{t=0} \rho(\sigma[tx]) 
=\frac{d}{dt}\Big \vert_{t=0} \int_{\Delta_1}\sigma[tx]^*\Phi_V^{(1)}\\
&=\frac{d}{dt}\Big \vert_{t=0} \int_0^1 t \exp(st\rho(L_x)) \circ \rho(i_{x})ds\\
&=\rho(i_x). &\qedhere
\end{align*}
\end{proof}

\begin{lemma} \label{right}
The functor $\Dscr$ is right inverse to $\Iscr$. Indeed,
\[\Iscr \circ  \Dscr = \id_{\Mod(\uC_{\sbullet}(G))}.\]
\end{lemma}

\begin{proof}
Consider a representation $\rho\colon \uC_{\sbullet}(G) \rightarrow \End(V) \in \Mod(\uC_{\sbullet}(G))$. Theorem \ref{A} implies that there exists a differential form 
$\Phi \in \Omega^{\sbullet}(G, \End(V))$ such that
\[ \rho(\sigma)=\int_{\Delta_k} \sigma^*\Phi.\]
Lemma \ref{mult} implies that $\Phi$ is a representation form. Let $\Psi$ be the representation form associated to $\Iscr(\Dscr (V))$.
Since $\Dscr(\Iscr(\Dscr (V)))= \Dscr(V)$ we conclude that $\Phi$ and $\Psi$ coincide in degrees $0$ and $1$. Since both are representation forms, Lemma \ref{injective} implies that they are equal. We conclude that $\Iscr(\Dscr (V))=V$.
\end{proof}

Putting together Lemma \ref{left} and Lemma \ref{right} we conclude the following.

\begin{theorem}\label{main}
Let $G$ be a simply connected Lie group with Lie algebra $\gfrak$. The functors 
$$
\Iscr \colon \Rep(\TT\gfrak)\rightarrow \Mod(\uC_{\sbullet}(G))
$$
and
$$
\Dscr\colon\Mod(\uC_{\sbullet}(G)) \rightarrow \Rep(\TT\gfrak)
$$
are equivalences of monoidal categories which are inverses to one another.
\end{theorem}

\subsection{The case where $G$ is not simply connected}

So far we have only considered simply connected Lie groups. We will describe a version of Theorem \ref{main} in the case of Lie groups that are not necessarily simply connected. 
\begin{definition}
Let $G$ be a Lie group with Lie algebra $\gfrak$. The category $\Rep_G(\TT\gfrak)$ is the full subcategory of $\Rep(\TT\gfrak)$ that consists of representations $\rho\colon \TT\gfrak \rightarrow \End(V)$ such that the corresponding representation of $\gfrak$ can be integrated to a representation of $G$.
\end{definition}
Of course, in the case where $G$ is simply connected, the category  $\Rep_G(\TT\gfrak)$ is equal to $\Rep(\TT\gfrak)$.

\begin{lemma}
Let $G$ be a connected Lie group with universal cover $\pi\colon \widetilde{G} \rightarrow G$ and denote by $\iota\colon Z \hookrightarrow \widetilde{G}$ the inclusion of the kernel of $\pi$. The pull-back functor
\[ (\pi_*)^*\colon  \Mod(\uC_{\sbullet}(G)) \rightarrow  \Mod(\uC_{\sbullet}(\widetilde{G}))\]
is fully faithful and identifies the category $\Mod(\uC_{\sbullet}(G))$ with the subcategory of $\Mod(\uC_{\sbullet}(\widetilde{G}))$
that consists of modules $V$ such that $(\iota_*)^*(V)$ is trivial.
\end{lemma}

\begin{proof}
Let us first show that the the functor $(\pi_*)^*$ is fully faithful. Take  representations $\rho\colon \uC_{\sbullet}(G) \rightarrow \End(V)$ and $\rho'\colon \uC_{\sbullet}(G) \rightarrow \End(V'))$. We need to show that a linear map $\phi\colon V \rightarrow V'$ commutes with the action of $\uC_{\sbullet}(G)$ if an only if it commutes with the action of $\uC_{\sbullet}(\widetilde{G})$. One direction follows from the functoriality. For the other direction, suppose that $\phi\colon V \rightarrow W$ commutes with the action of $\uC_{\sbullet}(\widetilde{G})$. Consider a simplex $\sigma\colon \Delta_k \rightarrow G$. Since $\Delta_k$ is simply connected and $\pi$ is a covering map, there is a simplex $\widetilde{\sigma}\colon \Delta_k \rightarrow \widetilde{G}$ such that $\pi \circ \widetilde{\sigma}=\sigma$. Then
\begin{eqnarray*}
\rho'(\sigma)(\phi(v))&=&\rho'( \pi \circ \widetilde{\sigma})(\phi(v))= (\pi^*\rho')(\widetilde{\sigma})(\phi(v))=\phi( (\pi^*\rho)(\widetilde{\sigma})(v))\\
&=&\phi( \rho(\pi \circ \widetilde{\sigma})(v)) =\phi( \rho(\sigma)(v)).
\end{eqnarray*}
Since $\sigma$ is arbitrary, we conclude that $\phi$ commutes with the action of $\uC_{\sbullet}(G)$. It remains to compute the image of $(\pi_*)^*$. Clearly, a representation of the form $\pi^*\rho$ vanishes on $\uC_{\sbullet}(Z)$. 

Conversely, suppose that $\widetilde{\rho}\colon \uC_{\sbullet}(\widetilde{G}) \rightarrow \End(V)$ is such that $\iota^*\rho$ is trivial. We need to show that $\widetilde{\rho}= \pi^*\rho$. We define $\rho\colon \uC_{\sbullet}(G) \rightarrow \End(V)$ by
\[ \rho(\sigma) = \widetilde{\rho}(\widetilde{\sigma}),\]
where $\widetilde{\sigma}$ is any lift of $\sigma$.  If $\widetilde{\sigma}'$ is a different lift of $\sigma$ then there exists $g \in Z$ such that $\widetilde{\sigma}'= \widetilde{\sigma} g$ and therefore
\[ \widetilde{\rho}(\widetilde{\sigma}')=\widetilde{\rho}(\widetilde{\sigma}) \circ \widetilde{\rho}(g)=\widetilde{\rho}(\widetilde{\sigma}).\]
We conclude that $\rho$ is well defined. One immediately checks that $\rho$ is a representation and, by construction, $\pi^*\rho= \widetilde{\rho}$.
\end{proof}

\begin{theorem}\label{nosc}
Let $G$ be a connected Lie group with Lie algebra $\gfrak$. The functor
$\Dscr\colon\Mod(\uC_{\sbullet}(G)) \rightarrow \Rep_G(\TT\gfrak)$ is an equivalence of categories.
\end{theorem}

\begin{proof}
By Theorem \ref{main} we know that 
\[ \widetilde{\Dscr}:  \Mod(\uC_{\sbullet}(\widetilde{G})) \rightarrow \Rep(\TT\gfrak)\]
is an isomorphism of categories. By naturality of the construction we know that
\[ \widetilde{\Dscr} \circ (\pi_*)^*=\Dscr.\]
Since both $\widetilde{\Dscr}$ and $\pi^*$ are fully faithful, so is $\Dscr$. Let us show that $\Dscr$ is essentially surjective. Take $V$ in $\Rep_G(\TT\gfrak)$. We know that there exists $W \in \Mod(\uC_{\sbullet}(\widetilde{G}))$ such that $\widetilde{\Dscr}(W)=V$. Since the representation of $\gfrak$ associated to $V$ can be integrated to $G$ we know that $(\iota_*)^*(W)$ is trivial. This implies that $W=(\pi_*)^*(W')$. Therefore $V= \Dscr(W')$.
\end{proof}


\section{The left adjoint to the forgetful functor}\label{examples}
There is a forgetful functor
\[ F\colon\Rep(\TT\gfrak) \rightarrow \Rep(\gfrak)\]
which is the restriction along the inclusion $\iota\colon \gfrak \hookrightarrow \TT\gfrak$.  This forgetful functor admits a left adjoint $U\colon \Rep(\gfrak) \rightarrow \Rep(\TT\gfrak)$ defined as follows. Given a representation $V$ of $\gfrak$, $U(V)$ is the cochain complex $\uC_{\sbullet}(\gfrak,V)$ with action of $\TT\gfrak$ given by
\[ \rho(L_x)( x_1\wedge \cdots \wedge x_k \otimes v)=\sum_{i=1}^k x_1 \wedge \cdots \wedge [x, x_i] \wedge \cdots x_k \otimes v+ x_1\wedge \cdots \wedge x_k \otimes \rho(x)(v),\]
and
\[ \rho(i_x)( x_1\wedge \cdots \wedge x_k \otimes v)=x\wedge x_1\wedge \cdots \wedge x_k \otimes v.\]

\begin{proposition}\label{K}
The functor $U\colon \Rep(\gfrak) \rightarrow \Rep(\TT\gfrak)$ is left adjoint to $F\colon \Rep(\TT\gfrak) \rightarrow \Rep(\gfrak)$. Moreover, $U$ is faithful.
\end{proposition}

\begin{proof}
It is obvious that $\rho([L_x,L_y])= [\rho(L_x),\rho(L_y)]$ and that $[\rho(i_x), \rho(i_y)]=0$. To show that $[\rho(L_x), \rho(i_y)]=\rho(i_{[x,y]})$ one computes
\begin{eqnarray*}
&&\!\!\!\!\!\!\!\!\!\!\!\!\!\!\!\!\!\!\!\!\!\!\!\!\!\!\!\!\!\!\!\![\rho(L_x), \rho(i_y)](x_1\wedge \dots \wedge x_k \otimes v )\\
&=&\rho(L_x)\circ \rho(i_y)(x_1\wedge \dots \wedge x_k \otimes v )-\rho(i_y)\circ \rho(L_x)(x_1\wedge \dots \wedge x_k \otimes v )\\
&=&\rho(L_x)(y\wedge x_1\wedge \dots \wedge x_k \otimes v )-  \sum_{i=1}^k y\wedge x_1 \wedge \dots \wedge [x, x_i] \wedge \dots x_k \otimes v\\
&&- y \wedge x_1 \wedge \dots \wedge x_k \wedge \rho(x)(v)\\
&=&[x,y]\wedge x_1\wedge \dots \wedge x_k \otimes v +\sum_{i=1}^k y\wedge x_1 \wedge \dots \wedge [x, x_i] \wedge \dots x_k \otimes v\\
&&+ y \wedge x_1 \wedge \dots \wedge x_k \wedge \rho(x)(v) -\sum_{i=1}^k y\wedge x_1 \wedge \dots \wedge [x, x_i] \wedge \dots x_k \otimes v\\
&&- y \wedge x_1 \wedge \dots \wedge x_k \wedge \rho(x)(v)\\
&=&\rho(\iota_{[x,y]}) x_1 \wedge \dots \wedge x_k \otimes v.
\end{eqnarray*}
Let us show that $\delta(\rho(i_x))= \rho(L_x)$ by computing
\begin{eqnarray*}
&&\!\!\!\!\!\!\!\!\!\!\!\!\!\!\!\!\!\!\!\!\!\!\!\!\delta(\rho(i_x))(x_1 \wedge \cdots \wedge x_k \otimes v)\\
&=&\delta (x \wedge x_1 \wedge \cdots \wedge x_k \otimes v)+ x\wedge \delta(x_1 \wedge \cdots \wedge x_k \otimes v)\\
&=& \sum_{i<i}(-1)^{i+j+1} [x_i,x_j]\wedge x \wedge \cdots \wedge \widehat{x}_i \wedge \cdots \wedge \widehat{x}_j \wedge \dots x_k \otimes v\\
&&+ \sum_{i=1}^k (-1)^{i+1}[x,x_i]\wedge x_1 \wedge \cdots \wedge \widehat{x}_i \wedge \cdots \wedge x_k \otimes v + x_1 \wedge \cdots \wedge x_k \otimes \rho(x)(v)\\
&&+ \sum_i (-1)^{i}x\wedge x_1 \wedge \dots \wedge \widehat{x}_i \wedge \cdots \wedge x_k \otimes \rho(x_i)(v)\\
&&+ \sum_{i<i}(-1)^{i+j+1}x\wedge [x_i,x_j] \wedge \cdots \wedge \widehat{x}_i \wedge \cdots \wedge \widehat{x}_j \wedge \dots x_k \otimes v\\
&&+ \sum_i (-1)^{i+1}x\wedge x_1 \wedge \cdots \wedge  \widehat{x}_i \wedge \cdots \wedge x_k \otimes \rho(x_i)(v) \\
&=&\sum_{i=1}^k x_1 \wedge \cdots \wedge [x, x_i] \wedge \cdots \wedge x_k \otimes v+ x_1\wedge \cdots \wedge x_k \otimes \rho(x)(v)\\
&=& \rho(L_x) (x_1 \wedge \cdots \wedge x_k \otimes v).
\end{eqnarray*}
It is clear that $U$ is faithful. Let us see that it is left adjoint to $F$. There is a natural map
\[ \Hom_{\Rep(\gfrak)}(K(V),V') \rightarrow \Hom_{\Rep(\TT\gfrak)}(V, F(V')),\]
that sends $\phi$ to $\phi_0$, its restriction to $\uC_0(\gfrak,V)=V$. This map is injective, since
\begin{equation} \label{x}
\phi( x_1\wedge \dots \wedge x_k \otimes v)=\phi( i_{x_1} \circ \dots \circ i_{x_k} (v)= i_{x_1} \circ \dots \circ i_{x_k} (\phi_0(v)). 
\end{equation}
Equation \eqref{x} also provides a way to extend any $\gfrak$ equivariant map from $V$ to $F(V')$ to a $\TT\gfrak$ map from $U(V)$ to $V'$. 
\end{proof}

\begin{remark}
Let us denote by $\Rep^0(\gfrak)$ the full subcategory of $\Rep(\gfrak)$ which consists of vector spaces concentrated in degree zero. The restriction of $U$ to $\Rep^0(\gfrak)$ is fully faithful.
\end{remark}

The relationship between the functors we have discussed and the standard functors in Lie theory is described in the following commutative diagram
\[ \xymatrix{
\Rep(G) \ar[r]^-{\mathsf{Lie}}&  \Rep(\gfrak)\ar[r]^{\mathsf{Lie}^{-1}}&\Rep(G)\\
\Mod(\uC_{\sbullet}(G)) \ar[r]^-{\Dscr}\ar[u]^-{F}&\ar[u]^-{F}  \Rep(\TT\gfrak)\ar[r]^-{\Iscr} &\ar[u]^-{F}\Mod(\uC_{\sbullet}(G)) \\
\ar[u]^-{U}\Rep(G)\ar[r]^{\mathsf{Lie}} & \ar[u]^-{U} \Rep(\gfrak)\ar[r]^{\mathsf{Lie}^{-1}}&\ar[u]^-{U} \Rep(G).}\]

 There is also a functor
\[ E: \Rep(\gfrak) \rightarrow \Rep(\TT\gfrak)\]
which sends $V$ to $\CE(\gfrak,V)$, which is a $\TT\gfrak$ module with structure given by
\[L_x( \xi \otimes v)=L_x\xi \otimes v+ \xi \otimes L_x v, \quad
\iota_x(\xi \otimes v)= \iota_x\xi\otimes v.\]
Clearly, $[L_x, L_y]=L_{[x,y]}$ and $[i_x,i_y]=0$. Let us show that $[L_x,i_y]=i_{[x,y]}$. We compute
\begin{eqnarray*}
[L_x,i_y](\xi \otimes v)&=&L_x (i_y\xi \otimes v) - i_y( L_x\xi \otimes v + \xi \otimes L_x v)\\
&=&L_x i_y \xi \otimes v+ i_y\xi \otimes L_x v - i_y L_x \xi \otimes v -\iota_y\xi \otimes L_xv\\
&=&\iota_{[x,w]}(\xi \otimes v).
\end{eqnarray*}
It remains to show that $[\delta ,i_x]=L_x$.  Since both operators are derivations with respect to the $\CE(\gfrak)$-module structure, it suffices to check the identity for $\vert \xi \vert \in \{0,1\}$. For $v$ in $V$ we compute
\[ [\delta, i_x](v)=i_x(\delta(v))=L_x(v).\]
Finally, for $\xi \in \CE^1(\gfrak)$ we compute
\begin{eqnarray*}
 [\partial, i_x](\xi \otimes v)(y)&=&\partial (i_x\xi \otimes v)(y) + i_x(\delta(\xi \otimes v))(y)\\
 &=&\xi(x)L_y v-\delta(\xi \otimes v)(x,y)\\
  &=&\xi(x)L_y v- \xi([x,y])\otimes v -\xi(x)L_y v+\xi(y)L_x v\\
  &=&L_x(\xi\otimes v)(y).
 \end{eqnarray*}

\begin{remark}
The restriction of $E$ to $\Rep^0(\gfrak)$ is fully faithful.
\end{remark}

The relationship between the present constructions and the standard functors in Lie theory is described in the following commutative diagram
\[ \xymatrix{
\Rep(G) \ar[r]^-{\mathsf{Lie}}&  \Rep(\gfrak)\ar[r]^-{\mathsf{Lie}^{-1}}&\Rep(G)\\
\Mod(\uC_{\sbullet}(G)) \ar[r]^-{\Dscr}\ar[u]^-{F}&\ar[u]^-{F}  \Rep(\TT\gfrak)\ar[r]^-{\Iscr} &\ar[u]^-{F}\Mod(\uC_{\sbullet}(G)) \\
\ar[u]^-{E}\Rep(G)\ar[r]^{\mathsf{Lie}} & \ar[u]^-{E} \Rep(\gfrak)\ar[r]^-{\mathsf{Lie}^{-1}}&\ar[u]^-{E} \Rep(G).}\]


\appendix
\section{Synthetic geometry and the image of the De Rham map}\label{A1}
We will adapt results  from \cite{Kock}  to characterize those singular cochains on $M$ that are given by integration of differential forms.

\begin{definition}
Let $M$ be a manifold. The vector space $\uC^{\square}_k(M)$ is the real vector space generated by smooth maps:
\[ \theta: I^k \rightarrow M. \]
The vector space $\uC^k_{\square}(M)$ is the vector space dual to $\uC^{\square}_k(M)$. For each $i=1,\dots, k$ and $a=0,1$
there is a map
\[ \delta^a_i\colon I^{k-1} \rightarrow I^k,\quad  (t_1, \dots,t_{k-1})=(t_1,\dots , t_{i-1},a, t_i, \dots ,t_{k-1}).\]
The linear map
\[ \delta\colon \uC^{k-1}_{\square}(M) \rightarrow \uC^k_{\square}(M), \quad \delta= \sum_{i=1}^k (-1)^i\Big((\delta^0_i)^*- (\delta^1_i)^*\Big)\]
satisfies $\delta^2=0$. The cubical complex of $M$ is the cochain complex $(\uC^{\sbullet}_{\square}(M),\delta)$.
\end{definition}

For each $k \geq1$, $i=1,\dots, k$, and  $s \in [0,1]$ we define maps:
\[ \alpha_i(s)\colon I^{k} \rightarrow I^k,\quad  (t_1, \dots, t_k) \mapsto (t_1, \dots, st_i, \dots ,t_k)\]
and
\[ \beta_i(s)\colon I^k \rightarrow I^k, \quad (t_1, \dots, t_k) \mapsto (t_1, \dots, (1-s) - st_i, \dots ,t_k)\]

\begin{definition}
We say that a cubical cochain $c \in \uC^k_{\square}(M)$  is subdivision invariant if it satisfies
\[ c(\theta) = c(\theta \circ \alpha_i(s)) + c( \theta \circ \beta_i(1-s))\]
for all $\theta\colon I^k \rightarrow M$, $i=1,\dots, k$, and all $s \in [0,1]$.
\end{definition}

\begin{definition}
A cubical cochain $c \in \uC^k_\square(M)$ is called alternating if for any permutation $\chi \in \mathfrak{S}_k$ and any $\theta\colon I^k \rightarrow M$, one has:
\[ c(\theta \circ \chi_*)= \mathrm{sgn}(\chi) c(\theta),\]
where $\chi_*\colon I^k \rightarrow I^k$ is the map $(t_1, \dots, t_k) \mapsto (t_{\chi(1)},\dots, t_{\chi(k)})$.
\end{definition}

The following result from \cite{Kock} characterizes those cubical cochains which are given by integration of 
a differential form.

\begin{theorem}[Theorem 3.4.4 of \cite{Kock}]\label{Kock}
Let $c \in \uC_{\square}^k(M)$ be a cubical cochain which is alternating and subdivision invariant. There exists a unique differential form $\omega  \in \Omega^k(M)$ such that:
\[ c(\theta)= \int_{I^k} \theta^*\omega,\]
for all maps $\theta\colon I^k \rightarrow M$. Moreover, this integration map is a bijection between the set of 
differential $k$-forms on $M$ and the set of alternating, subdivision invariant cubical cochains.
\end{theorem}

There is a natural map $\tau\colon \uC^k(M) \rightarrow \uC^k_{\square}(M)$ given by:
\[ \tau(c) (\theta) = \sum_{\chi \in \mathfrak{S}_k}  \mathrm{sgn}(\chi)c( \theta \circ \chi_*),\]
with $\chi_*\colon \Delta_k \rightarrow I^k$ as before. 

Recall that a finite simplicial complex is a finite set $X$ together with a subset $S \subseteq \mathcal{P}(X)$ of simplices that satisfies:
\begin{itemize}
\item If $s \in S$ then any subset of $s$ is also in $S$.
\item If $s$ and $s'$ are in $S$ then $s \cap s' \in S$.
\end{itemize}
There is a functor $\Delta\colon \mathbf{FSets} \rightarrow \mathbf{Top}$ from the category of finite sets to the category of topological spaces
given by 
\[\Delta(y)=\left\{ \sum_{y \in y} t_y y \mid t_y \geq 0, \sum_{y \in y} t_y=1\right\}.\]
Clearly $\Delta(y)$ is isomorphic to a simplex of dimension $d=\vert y\vert$. The geometric realization of a finite simplicial complex $X$ is the topological space
\[ \vert X\vert=\coprod_{s\in S}  \Delta(s)/{\sim},\]
where $\sim$ is the equivalence relation generated by the condition that if $\iota\colon s' \hookrightarrow s$ is an inclusion then
\[ \Delta(\iota)(x)\sim x.\]
A triangulation of a manifold is a simplicial complex $X$ together with a homeomorpohism $f \colon\vert X \vert \rightarrow M$. There is a natural projection map $\pi\colon \coprod_{s\in S}  \Delta(s) \rightarrow \vert X\vert$. A $k$-simplex of the triangulation $f$ is a map
$\pi^*f\colon \Delta(s) \rightarrow M$ where $s \in S$ is a simplex with $k+1$ elements. We will denote by $\Delta_k(f)$ the set of all simplices of the triangulation. We say that a triangulation is smooth if each of its simplices is a smooth map. We are interested in the special case where $M=I^k$. In this case we say that a triangulation is affine if each of the simplices of the triangulation is an affine map. We will assume that the set $X$ is totally ordered so that each of the spaces $\Delta(s)$ comes with a natural orientation.

\begin{definition}
Let $M$ be a manifold and $c \in \uC^{\sbullet}(M)$ a singular cochain .
We will say that $c$ is subdivision invariant if for any affine triangulation $f\colon \vert X\vert \rightarrow I^k$ of the cube and any smooth map $\theta\colon I^k \rightarrow M$, it satisfies
\[\tau(c)(\theta)=\sum_{\sigma \in \Delta_k(f)} \mathrm{sgn}(\sigma) c(\theta \circ \sigma),\]
where $\mathrm{sgn}(\sigma)$ is plus or minus one depending on whether $\sigma$ preserves or reverses the orientation.
\end{definition}

\begin{lemma}\label{division}
Let $c \in \uC^k(M)$ be a singular cochain in $M$ which is subdivision invariant. Then $\tau(a) \in \uC^k_{\square}(M)$ is also subdivision invariant.
\end{lemma}

\begin{proof}
Fix a map $\theta\colon I^k \rightarrow M$. We need to prove that:
\[ \tau(c)(\theta) = \tau(c) (\theta \circ \alpha_i(s)) + \tau(c)( \theta \circ \beta_i(1-s)).\]
Let us fix a triangulation $f\colon \vert X\vert \rightarrow I^k$ of the cube that restricts to triangulations $\widehat{f}$ of $I^{i-1} \times [0,s]\times I^{k-i}$ and $\widetilde{f}$ of $I^{i-1} \times [s,1]\times I^{k-i}$. Then
\begin{eqnarray*}
\tau(c)(\theta)&=&\sum_{\sigma \in \Delta_k(f)} \mathrm{sgn}(\sigma) c(\sigma) \\
&=& \sum_{\sigma \in \Delta_k(\widehat{f})} \mathrm{sgn}(\sigma) c(\sigma)+ \sum_{\sigma \in \Delta_k(\widetilde{f})} \mathrm{sgn}(\sigma) c(\sigma)\\
&=& \tau(c) (\theta \circ \alpha_i(s)) + \tau(c)( \theta \circ \beta_i(1-s)). \qedhere
\end{eqnarray*}
\end{proof}

\begin{definition}
A singular cochain $c \in \uC^k(M)$ is alternating if $\tau(c)$ is alternating in $\uC^k_{\square}(M)$.
\end{definition}

\begin{definition}
For each $k\geq 1$ we define a map:
\[ P_k\colon I^k \rightarrow \Delta_k, \quad (t_1, \dots, t_k) \mapsto (y_1, \dots, y_k),\]
where
\[ y_i=\max(t_i, \dots, t_k).\]
\end{definition}

\begin{lemma}\label{chi}
Let $\iota\colon \Delta_k \rightarrow I^k$ be the inclusion map. Then
\[ P_k \circ \iota =\id_{\Delta_k}.\]
Moreover, given any permutation $\chi \in \mathfrak{S}_k$ which is different from the identity, the image of the map $P_k \circ \sigma_\chi$ lies on the boundary of $\Delta_k$.
\end{lemma}

\begin{proof}
Clearly,
\[ P_k \left( \iota (t_1, \dots, t_k) \right)=(t_1,\dots, t_k).\]
On the other hand, if $\chi$ is different from the identity then any point in the image of $P_k \circ \sigma_\chi$ will have
at least two coordinates which are equal.
\end{proof}

\begin{theorem}\label{A}
Let $c \in \uC^k(M)$ be a singular cochain which is alternating and subdivision invariant and vanishes on thin simplices. Then, there exists a unique differential form $\omega  \in \Omega^k(M)$ such that:
\[ c(\sigma)= \int_{\Delta_k} \sigma^*\omega,\]
for all simplices $\sigma\colon \Delta_k \rightarrow M$.
\end{theorem}

\begin{proof}
By Lemma \ref{division} we know that $\tau(c)$ is subdivision invariant. Since $\tau(c)$
is alternating, Theorem \ref{Kock} implies that there is a differential form $\omega \in \Omega^k(M)$ such that
\[ \tau(a)(\theta)=\int_{I^k} \theta^*\omega,\]
for all $\theta\colon I^k \rightarrow M$. We claim that for any $\sigma\colon \Delta_k \rightarrow M$ one has:
\[ c(\sigma)= \int_{\Delta_k} \sigma^*\omega.\]
Set $\theta = \sigma \circ P_k$.  Then
\begin{eqnarray*}
 \int_{I^k} \theta^*\omega= \sum_{\chi \in \mathfrak{S}_k}\int_{\Delta_k} \mathrm{sgn}(\chi)(\theta \circ \sigma_\chi)^*\omega= \sum_{\chi \in  \mathfrak{S}_k} \int_{\Delta_k}\mathrm{sgn}(\chi)(\sigma \circ P_k \circ \sigma_\chi)^*\omega =\int_{\Delta_k} \sigma^*\omega.
\end{eqnarray*}
Where, we have used Lemma \ref{chi} to conclude that only the integral associated to the identity is nonzero, because all other simplices are thin. On the other hand,
\begin{eqnarray*} 
\int_{I^k} \theta^*\omega&=& \tau(c)(\theta)= \sum_{\chi \in  \mathfrak{S}_k} \mathrm{sgn}(\chi) c(\sigma \circ P_k \circ \sigma_\chi)=c(\sigma).
\end{eqnarray*}
Where, again, we have used Lemma \ref{chi} and the fact that $c$ vanishes on thin simplices to conclude that only the integral associated to the identity is nonzero. This shows that $c$ is given by integrating $\omega$ as required. The uniqueness of $\omega$ is clear.
\end{proof}


\end{document}